\documentclass[lettersize,journal]{IEEEtran}
\usepackage{amssymb,epsfig,color,cite,amsmath,amsfonts,mathrsfs,algorithm}
\usepackage{mathtools}
\usepackage{epstopdf}
\usepackage{datetime,fancyhdr}
\usepackage{algpseudocode}
\usepackage{arydshln}
\usepackage{multirow}
\usepackage{pifont}
\usepackage{amsthm} 
\allowdisplaybreaks

\usepackage{times} 
\usepackage{epsfig}
\usepackage{graphicx}
\usepackage{latexsym}
\usepackage{subfigure}

\newtheorem{lemma}{Lemma}[section]
\newtheorem{theorem}{Theorem}[section]

\newtheorem{remark}{Remark}[section]
\newtheorem{definition}{Definition}[section]

\newtheorem{proposition}{Proposition}[section]
\newtheorem{assumption}{Assumption}[section]

\ifCLASSINFOpdf
\else
\fi

\begin{document}

\title{Efficient Single-Loop Stochastic Algorithms for Nonconvex-Concave Minimax Optimization}

\author{Xia Jiang, Linglingzhi Zhu, Taoli Zheng, and Anthony Man-Cho So,~\IEEEmembership{Fellow,~IEEE}
\thanks{X. Jiang (xiajiang@cuhk.edu.hk), L. Zhu (llzzhu@se.cuhk.edu.hk), T. Zheng (tlzheng@se.cuhk.edu.hk), and A. M.-C. So (manchoso@se.cuhk.edu.hk) are with Department of Systems Engineering and Engineering Management at The Chinese University of Hong Kong, Hong Kong. This work is supported in part by the Hong Kong Research Grants Council (RGC) General Research Fund (GRF) project CUHK 14204823. }
}

\maketitle

\begin{abstract}
 Nonconvex-concave (NC-C) finite-sum minimax problems have wide applications in signal processing and machine learning tasks. Conventional stochastic gradient algorithms, which rely on uniform sampling for gradient estimation, often suffer from slow convergence rates and require bounded variance assumptions. While variance reduction techniques can significantly improve the convergence of stochastic algorithms, the inherent nonsmooth nature of NC-C problems makes it challenging to design effective variance reduction techniques. To address this challenge, we develop a novel probabilistic variance reduction scheme and propose a single-loop stochastic gradient algorithm called the probabilistic variance-reduced smoothed gradient descent-ascent (PVR-SGDA) algorithm. The proposed PVR-SGDA algorithm achieves an iteration complexity of $\mathcal{O}(\epsilon^{-4})$, surpassing the best-known rates of stochastic algorithms for NC-C minimax problems and matching the performance of state-of-the-art deterministic algorithms. Furthermore, to completely eliminate the need for full gradient computation and reduce the gradient complexity, we explore another variance reduction technique with auxiliary gradient trackers and propose a smoothed gradient descent-ascent algorithm without full gradient calculation, called ZeroSARAH-SGDA, for NC-C problems. The ZeroSARAH-SGDA algorithm achieves a comparable iteration complexity to PVR-SGDA, while reducing the gradient oracle calls at each iteration. Finally, we demonstrate the effectiveness of the proposed two algorithms through numerical simulations.

\end{abstract}


%
\IEEEpeerreviewmaketitle

\section{Introduction}
%
%
%
%
\IEEEPARstart{I}{n} recent years, various applications in decentralized optimization \cite{Distributed_Sadd}, (distributionally) robust optimization \cite{ben2009robust,delage2010distributionally,kuhn2019wasserstein}, and reinforcement learning \cite{littman1994markov,dai2018sbeed,zhang2021multi} have underscored the need to tackle nonconvex-concave (NC-C) smooth minimax problems.
While the ultimate objective is to train models that perform well to unseen data, in practice, we deal with a finite dataset during training. This leads to the finite-sum minimax problem addressed in this paper:
\begin{align}\label{opti_pro}
\min_{x\in \mathcal{X}}\max_{y\in \mathcal{Y}} f(x,y):=\frac{1}{n}\sum_{i=1}^n f_i(x,y),
\end{align}
where $f:\mathbb{R}^n\times \mathbb{R}^d\to \mathbb{R}$ can be nonconvex with respect to $x$ but concave with respect to $y$, $f_i$ refers to the cost function associated with the $i$-th sample of a finite training dataset, and $\mathcal{X}\subseteq \mathbb{R}^n$, $\mathcal{Y}\subseteq \mathbb{R}^d$ are nonempty convex compact sets. 
Decentralized nonconvex optimization problems prevalent in multi-agent networks, such as smart grids, UAV swarms, and intelligent transportation systems, can be systematically reformulated into NC-C minimax problems through Lagrangian duality theory \cite{Distributed_Sadd}.
Consequently, developing efficient algorithms for NC-C minimax problems is essential to address the optimization challenges inherent in various scenarios.



Stochastic first-order algorithms have attracted significant research interest for solving minimax problems due to their scalability and efficiency. Among these, the stochastic gradient descent-ascent (StocGDA) algorithm is the most widely used, extending the stochastic gradient descent approach to the minimax setting. For NC-C cases, the work \cite{SGDA} introduced a stochastic variant of the two-timescale GDA algorithm \cite{two_GDA}, which employs unequal step sizes and provides a non-asymptotic convergence guarantee.
On another front, multi-loop type algorithms with acceleration in the subproblems \cite{yang2020catalyst} have advantages over GDA variants in terms of iteration complexity for general NC-C problems. Both of these algorithms emphasize the importance of updating $y$ more frequently than $x$ for solving minimax problems, while the two-timescale StocGDA is relatively easier to implement and generally demonstrates superior empirical performance compared to the multi-loop type algorithms.


To further capitalize on the performance advantages of two-timescale StocGDA, the work \cite{alter_prox_nc_c} explored the favorable convergence properties of the alternating version of the two-timescale GDA by proposing a stochastic alternating proximal gradient algorithm. Additionally, the work \cite{zhang2022sapd} introduced a stochastic algorithm based on the inexact proximal point method with unequal step sizes to address NC-C minimax problems.  
Both stochastic algorithms presented in \cite{alter_prox_nc_c,zhang2022sapd} 
require an iteration complexity of $\mathcal{O}(\epsilon^{-6})$ to solve NC-C minimax problems.

While these stochastic approaches have established non-asymptotic convergence guarantees for NC-C minimax problems, they generally exhibit significantly slower convergence rates compared to deterministic methods and require the additional assumption of bounded gradient variance. One effective technique to enhance the convergence performance of stochastic optimization algorithms is variance reduction. For nonconvex minimization problems, extensive research has demonstrated the effectiveness of variance reduction techniques in improving computational complexity and achieving faster convergence rates, with notable examples including the stochastic path integrated differential estimator (SPIDER) \cite{spider_paper} and the stochastic recursive momentum (STORM) algorithm \cite{NEURIPS2019_b8002139}.
In the realm of minimax optimization, 
the works \cite{zhang2022sapd,SREDA_nips2020} introduced double-loop variance-reduced stochastic algorithms based on the SPIDER technique for nonconvex-strongly concave (NC-SC) minimax optimization, which necessitate large batch sizes during each periodic iteration. 
To the best of our knowledge, there are currently no efficient variance-reduced stochastic algorithms with low iteration complexity available for addressing the NC-C minimax problems to
enhance the convergence performance of StocGDA algorithms.

Compared to NC-SC minimax problems, NC-C minimax problems are more challenging to solve due to the nonsmoothness introduced by the nonunique dual solutions resulting from concavity. Classic variance reduction methods, such as SVRG and SAGA, are not applicable to nonsmooth loss functions \cite{pmlr-v139-song21d}. Additionally, techniques like SPIDER and STORM suffer from unavoidable variance in the stochastic gradient estimators, which impedes the assurance of recursive gradient descent for NC-C minimax problems.

To address the challenge of nonsmoothness in gradient estimation, it is crucial to carefully select and analyze an appropriate variance reduction technique. The contributions of this paper are summarized as follows.
\begin{itemize}
    \item To eliminate variance terms in stochastic gradient estimators and ensure the recursive gradient descent property, we develop a probability-based gradient updating scheme and propose a novel probabilistic variance-reduced smoothed gradient descent-ascent (PVR-SGDA) algorithm for solving NC-C minimax problems. By integrating the probabilistic variance reduction technique with Moreau-Yosida smoothing acceleration, the proposed PVR-SGDA algorithm achieves an iteration complexity of $\mathcal{O}(\epsilon^{-4})$, where $\epsilon$ denotes the desired optimization accuracy. This result represents a significant improvement over the $\mathcal{O}(\epsilon^{-6})$ complexity of stochastic algorithms in \cite{zhang2022sapd} for NC-C minimax problems.
    \item To further reduce the gradient complexity and completely eliminate the need for full gradient computations in PVR-SGDA, we investigate a variance reduction technique known as ZeroSARAH, as introduced in \cite{Li2021ZeroSARAHEN}. Building on this approach, we propose another smoothed gradient descent-ascent (ZeroSARAH-SGDA) algorithm tailored for NC-C minimax problems. The proposed ZeroSARAH-SGDA algorithm matches the iteration complexity of PVR-SGDA while significantly reducing the gradient oracle calls.  However, this improvement comes at the expense of a  more intricate algorithmic framework, which necessitates additional storage space to manage and update auxiliary gradient tracking variables. 
    \item The two proposed stochastic algorithms not only achieve the best-known iteration complexity of deterministic counterparts for NC-C minimax optimization \cite{zhang2020single}, but also significantly reduce gradient computation costs through stochastic gradient estimation. In addition, in contrast to existing stochastic algorithms that rely on bounded gradient variance \cite{alter_prox_nc_c, SGDA}, as well as multi-loop variance-reduced stochastic algorithms \cite{SREDA_nips2020,yang2020catalyst}, both proposed stochastic algorithms  adopt the single-loop structure, which simplifies implementation, and remain robust to gradient variance.
\end{itemize}
\par The notation we use in this paper is standard. We use $[n]$ to denote the set $\{1,2,\ldots,n\}$ for any positive integer $n$. We use $I_d$ to denote a $d\times d$ identity matrix and $\otimes$ to denote the Kronecker product. Let the Euclidean space of all real vectors be equipped with the inner product $\langle x,y\rangle:=x^{\top}y$ for any vectors $x, y$ and $\|\cdot\|$ denote the induced norm. 
For a differentiable function $f$, the gradient of $f$ is denoted as $\nabla f$.
\section{Motivating Applications}
The nonconvex-concave minimax problems have wide applications in the control and optimization field. We provide two representative ones on signal processing and power control. 
 \subsection{Distributed optimization over multi-agent networks}
\par The popular optimal consensus problem and resource allocation problem in networked systems can be formulated as NC-C minimax problems \cite{Distributed_Sadd}. Consider a connected network graph $\mathcal{G}$ of $N$ agents, where each agent $i\in [N]$ only has access to local objective function $f_i$ and can communicate with its neighbors. The optimal consensus problem over multi-agent networks with set constraints is described by
\begin{align}\label{consus_pro}
\min _{x \in \Omega} \sum_{i=1}^N f_i(x_i),\ \ \text {s.t. }\left(\mathcal{L} \otimes I_n\right) x=0,
\end{align}
where $\mathcal{L} \in \mathbb{R}^{N \times N}$ is the Laplacian matrix of graph $\mathcal{G}$,
$x_i \in \Omega_i \subseteq \mathbb{R}^n$ with $\Omega=\prod_{i=1}^N \Omega_i$ being the Cartesian product of the local convex compact constraint sets $\Omega_i$ for $i\in [N]$, and $x=\operatorname{col}\left(x_i\right)_{i=1}^N \in \mathbb{R}^{n N}$. Each agent $i$ owns a local variable estimate $x_i$. Since the graph $\mathcal{G}$ is connected, $\left(\mathcal{L} \otimes I_n\right) x=0$ implies that the consensus condition $x_i=x_j$ holds for all $i, j \in [N]$.
\par The augmented Lagrangian function of the consensus problem \eqref{consus_pro} is $\mathscr{L}(x, v):=\sum_{i=1}^N f_i(x_i)+v^T\left(\mathcal{L} \otimes I_n\right) x+\frac{1}{2} x^T\left(\mathcal{L} \otimes I_n\right)x$, where $v:=\operatorname{col}\left(v_i\right)_{i=1}^N \in \mathbb{R}^{nN}$ is the dual variable.
Since $\mathscr{L}$ is an NC-C function,
 problem \eqref{consus_pro} can be transformed into the target constrained minimax problem
$
\min _{x \in \Omega} \max _{v \in \mathcal{V}} \mathscr{L}(x, v),
$
where the convex compact set $\mathcal{V}\subseteq\mathbb{R}^{n N}$ is chosen to be sufficiently large.
\subsection{Power control and transceiver design problem}
\par The wireless transceiver design problem involves $K$ transmitter-receiver pairs that transmit over $N$ channels to maximize their minimum rates. The transmitter $k$ transmits messages with power $x_k\triangleq \left[x_k^1 ; \cdots ; x_k^N\right]$, and its rate is given by:
    $R_k\left(x_1, \ldots, x_K\right)=\sum_{n=1}^N \log \left(1+\frac{a_{k k}^n x_k^n}{\sigma^2+\sum_{l=1, l \neq k}^K a_{l k}^n x_{l}^n}\right)$,
where $a_{l k}^n$ denotes the channel gain between the pair $(l, k)$ on the $n$th channel, and $\sigma^2$ is the noise power \cite{Lu_2020}. 
The function $R_k$ is a non-convex function on $x\triangleq \left[x_1 ; \cdots ; x_K\right]$. Let $\bar{x}$ denote the power budget for each user, then the classical max-min fair power control problem is: $\max _{x \in \mathcal{X}} \min _k R_k(x)$, where $\mathcal{X}\triangleq\{x \mid 0 \leq \sum_{n \in [N]} x_k^n \leq \bar{x}, \ \forall k\in [K]\}$ denotes the feasible power allocations. This max-min fair power control problem can be equivalently formulated as the following nonconvex-concave minimax problem (c.f. \cite[Section I.A]{Lu_2020}),
$$
\min _{x \in \mathcal{X}} \max _{y \in \Delta} \sum_{k=1}^K-R_k\left(x_1, \cdots, x_K\right) \times y_k,
$$
where the set $\Delta \subseteq \mathbb{R}^K$ is the standard simplex.

\section{Problem description and algorithm design}\label{solver_design}

For solving the general smooth NC-C problem \eqref{opti_pro}, a straightforward approach is to use the two-timescale GDA algorithms, 
and the work \cite{zhang2020single} further utilizes the Moreau-Yosida smoothing technique to accelerate them. To be specific, the smoothing technique introduces an auxiliary variable $z$ and defines a regularized function
\begin{align}\label{K_func}
K(x,z;y):=f(x,y)+\frac{r}{2}\|x-z\|^2.
\end{align}
The additional quadratic term smooths the primal update and facilitates a better trade-off between the primal and dual updates when running GDA on this regularized function.

Utilizing the regularized function \eqref{K_func}, we propose a stochastic gradient descent-ascent algorithm with a probabilistic variance reduction technique in the following Algorithm \ref{vr_agda}.
Here, the stochastic gradients of the function $K$ are given by
\begin{equation}\label{nabl_tilde_K_def}
\begin{aligned}
\nabla_x \tilde{K}(x,z;y)&:=G_x (x,y,\xi_1)+r(x-z),\\
\nabla_y \tilde{K}(x,z;y)&:=G_y(x,y,\xi_2),
\end{aligned}
\end{equation}
where $G_x(x,y,\xi_1)$ and $G_y(x,y,\xi_2)$ are stochastic estimators of $\nabla_x f(x,y)$ and $\nabla_y f(x,y)$ using random samples $\xi_1$ and $\xi_2$, respectively.
For simplicity, we denote $\nabla K_t:=\nabla K(x_t,z_t;y_t)$ and $\nabla \tilde{K}_t:=\nabla \tilde{K}(x_t,z_t;y_t)$ for $t\in\mathbb{N}$.

\begin{algorithm}
	\caption{Probabilistic Variance-Reduced Smoothed Gradient Descent-Ascent (PVR-SGDA)}
	\label{vr_agda}
	\begin{algorithmic}[1] 
		\State \textbf{Initialize:} $(x_0, y_0, z_0)$, step sizes $\eta_x > 0$, $\eta_y > 0$, $\rho > 0$, number of epochs $T$ 
		\For {$t = 0, \ldots, T-1$}
		\State $v_t = \begin{cases} 
		\nabla_x K_t & \text{with prob. } p \\
		v_{t-1} + \nabla_x \tilde{K}_t - \nabla_x \tilde{K}_{t-1} & \text{with prob. } 1-p
		\end{cases}$
		\State $w_t = \begin{cases} 
		\nabla_y K_t & \text{with prob. } p \\
		w_{t-1} + \nabla_y \tilde{K}_t - \nabla_y \tilde{K}_{t-1} & \text{with prob. } 1-p
		\end{cases}$
		\State $x_{t+1} = P_{\mathcal{X}}(x_t - \eta_x v_t)$
		\State $y_{t+1} = P_{\mathcal{Y}}(y_t + \eta_y w_t)$
		\State $z_{t+1} = z_t + \rho(x_{t+1} - z_t)$
		\EndFor
	\end{algorithmic}
\end{algorithm}
\par Next, we state some basic assumptions in this paper.
\begin{assumption}\label{f_assum}
	The function $f$ is differentiable and there exists a positive constant $L>0$ such that for all $x_1,x_2\in \mathcal{X}$ and $y_1,y_2\in \mathcal{Y}$, 
	\begin{align*}
	\|\nabla_x f(x_1,y_1)-\nabla_x f(x_2,y_2)\|&\leq L[\|x_1-x_2\|+\|y_1-y_2\|],\\
	\|\nabla_y f(x_1,y_1)-\nabla_y f(x_2,y_2)\|&\leq L[\|x_1-x_2\|+\|y_1-y_2\|].
	\end{align*}
\end{assumption}
\begin{assumption}\label{unbiased_g}
	The vectors $G_x(x,y,\xi_1)$ and $G_y(x,y,\xi_2)$ are unbiased stochastic gradient estimators of $\nabla_x f(x,y)$ and $\nabla_y f(x,y)$, respectively, i.e. $\mathbb{E}[\nabla \tilde{K}(x,z;y)]=\nabla K(x,z;y)$.
\end{assumption}
\begin{remark}
    Assumption \ref{unbiased_g} naturally holds when the samples $\xi_1$ and $\xi_2$ are chosen independently from an identical distribution. Assumption \ref{f_assum} is standard in minimax optimization. These two assumptions are commonly adopted in existing theoretical studies \cite{alter_prox_nc_c,zhang2022sapd,SREDA_nips2020}.
\end{remark}
 With the above assumptions, we assume $r > L$ in the rest of this paper, and then the regularized function $K$ owns the following important property.
 \begin{lemma}\label{K_sc_m}
The function $K(\cdot,z;y)$ is strongly convex with $r-L$ and $\nabla_x K(\cdot,z;y)$ is Lipschitz continuous with constant $L+r$.
 \end{lemma}

\section{Convergence analysis}\label{proof_sec}
This section concentrates on the convergence analysis of the proposed PVR-SGDA algorithm. For the analysis convenience, we provide some necessary notations as follows:
\begin{itemize}
    \item[1.] $h(x,z)\coloneqq \max_{y\in \mathcal{Y}} K(x,z;y)$;
    \item[2.] $d(y,z)\coloneqq\min_{x\in\mathcal{X}} K(x,z;y)$;
    \item[3.] $P(z)\coloneqq\max_{y\in \mathcal{Y}}\min_{x\in \mathcal{X}} K(x,z;y)$;
    \item[4.] $x(y,z)\coloneqq{\rm argmin}_{x\in\mathcal{X}} K(x,z;y)$;
    \item[5.] $x^*(z)\coloneqq{\rm argmin}_{x\in \mathcal{X}} h(x,z)$;
    \item[6.] $x_+(y,z)\coloneqq P_{\mathcal{X}}(x-\eta_x\nabla_x K(x,z;y))$;
    \item[7.] $ y_+(z)\coloneqq P_{\mathcal{Y}}(y+\eta_y\nabla_y K(x(y,z),z;y)$;
    \item[8.] $Y(z)\coloneqq{\rm argmax}_{y\in \mathcal{Y}} d(y,z)$ and $y(z)\in Y(z)$. 
\end{itemize}
Inspired by the works \cite{li2023nonsmooth,NEURIPS2023_a961dea4},  we propose a novel potential function as follows:
\begin{align}\label{phi_func_def}
\Phi_t\coloneqq V_t+\frac{\gamma}{2p} (\|\nabla_x K_t-v_t\|^2+\|\nabla_y K_t-w_t\|^2),
\end{align}
where $V_t\coloneqq K_t-2d(y_t,z_t)+2P(z_t)$ is bounded below and $\gamma>0$ is a constant parameter to be determined. Notably, due to the positivity of $\gamma$, the lower boundedness of $\Phi_t$ is guaranteed. Here, the first term $V_t$  in \eqref{phi_func_def} can be rewritten as 
\begin{align}\label{V_t_de}
V_t=\underbrace{K_t-d(y_t,z_t)}_{\text{Primal Descent}}+\underbrace{P(z_t)-d(y_t,z_t)}_{\text{Dual Ascent}}+\underbrace{P(z_t)}_{\text{Proximal Descent}}.
\end{align}
The potential function closely links the proximal function $P$ to the updates in the proposed algorithm on $K$, bridged by an ascent step on the dual function $d$. 
The second term in \eqref{phi_func_def} accounts for the error in the gradient estimate.

\subsection{Descent Property of $\Phi_t$}
\par We first analyze the descent property of the term $V_t$, which relies on the primal descent, dual ascent and proximal descent provided in Appendix \ref{V_des_proof_Sec}. 
\begin{lemma}\label{V_ineq_lem}
Suppose Assumptions \ref{f_assum} and \ref{unbiased_g}  hold. The function $V_t$ defined in \eqref{V_t_de} satisfies
   \begin{equation}
    \label{Evsub}
      \begin{aligned}
          V_t-V_{t+1}\geq\ &s_x\|x_{t+1}\!-x_t\|^2 \!+\!s_y\|y_{t+1}\!-y_t\|^2 \!+\!s_z\|z_{t+1}\! -z_t\|^2 \\
          &-s_v \|\nabla_x K_t-v_t\|^2-s_w\|\nabla_y K_t-w_t\|^2 \\
          &-2 L\left(1+24\omega^2r\rho\sigma_2^2\eta_y^2L\right)\|x_t-x_+(y_t,z_t)\|^2  \\
          &-24r\rho \|x^*(z_t)-x(y_+^t(z_t),z_t)\|^2,
    \end{aligned}  
   \end{equation} 
   where $\omega=\frac{\eta_x L+\eta_x r+1}{\eta_x r-\eta_x L}$, $s_x:= \frac{1}{\eta_x}-\frac{r+L+1}{2}-L^2$,  $s_y:=\frac{1}{\eta_y}-\frac{L+1}{2}-L(1+\omega)^2-L_d$, $s_z:= \frac{5r}{6\rho}-\frac{r}{2}-2r\sigma_1-48r\rho \sigma_1^2$, $s_v:=\frac{1}{2}+\frac{2 \eta_x^2 L}{(1+\omega)^2}$,  and $s_w=1+48r\rho \eta_y^2\sigma_2^2$.
\end{lemma}
Then, we provide the decent properties of the last two terms in \eqref{phi_func_def}, and the proof is provided in Appendix \ref{gradient_des_appendix}.
\begin{lemma}\label{lemma:key}
    Suppose Assumptions \ref{f_assum} and \ref{unbiased_g}  hold. The stochastic gradient estimators $v_t$ and $w_t$ generated by Algorithm \ref{vr_agda} satisfy
    \begin{align}
    &\mathbb{E}[\|\nabla_x K_t-v_t\|^2]-\mathbb{E}[\|\nabla_x K_{t+1}-v_{t+1}\|^2]\notag\\
    \geq\ & p\mathbb{E}[\|\nabla_x K_t-v_t\|^2]-3 \left(1-p\right) (L+r)^2\mathbb{E}[\|x_{t+1}-x_t\|^2]\notag\\
    & -3 \left(1-p\right)L^2 \mathbb{E}[\|y_{t+1}-y_t\|^2]\notag\\
    &-3 \left(1-p\right)r^2 \mathbb{E}[\|z_{t+1}-z_t\|^2],  \label{nablax_v}
    \end{align}
    and
    \begin{align}
    &\mathbb{E}[\|\nabla_y K_t-w_t\|^2]-\mathbb{E}[\|\nabla_y K_{t+1}-w_{t+1}\|^2]\notag\\
    \geq\ &p \mathbb{E}[\|\nabla_y K_t-w_t\|^2] -2\left(1-p\right) L^2\mathbb{E}[\|x_{t+1}-x_t\|^2]\notag\\
    & -2\left(1-p\right) L^2\mathbb{E}[\|y_{t+1}-y_t\|^2]. \label{nablay_w}
\end{align}
\end{lemma}
\par With the descent properties in Lemmas \ref{V_ineq_lem} and \ref{lemma:key}, the following proposition quantifies the change of  $\Phi_t$ after one round of updates, whose proof is provided in Appendix \ref{des_phi_pvr}.

\begin{proposition}\label{des_phi}
Suppose Assumptions \ref{f_assum} and \ref{unbiased_g}  hold. Without loss of generality, we set $L\geq 1$. Let $2L\leq r\leq 4L$, $\gamma = 4+\frac{2}{L}$, and $p\in (0,1]$. The step-sizes satisfy 
\begin{align*}
    \eta_x&\leq \frac{p}{p\left(1+24L+2L^2\right)+80L^2\gamma},\\
   \eta_y &\leq \min\left\{\frac{p}{2p\left(1+9L\right)+10\gamma L^2},\frac{1}{2L(1+\omega)^2}\right\},\\
    \rho &\leq \frac{4p}{1200p+9r \gamma},
\end{align*}
then for any $t\geq 0$,
    \begin{align*}
     &\mathbb{E}[\Phi_t-\Phi_{t+1}]\\
      \geq \ &\frac{1}{2\eta_x}\mathbb{E} [\|x_{t+1}-x_t\|^2]+\frac{1}{4\eta_y} \mathbb{E}[\|y_t-y_+^t(z_t)\|^2]\\
     &+\frac{r}{6\rho}\mathbb{E} [\|z_t-z_{t+1}\|^2] +\frac{\gamma}{4}\mathbb{E}[\|\nabla_x K_t-v_t\|^2]\\
     &+\!\frac{\gamma}{4}\mathbb{E}[\|\nabla_y K_t-w_t\|^2]\!-\!24r\rho  \mathbb{E}[\|x^*(z_t)\!-\!x(y_+^t(z_t),z_t)\|^2].
\end{align*}
\end{proposition}
We proceed to bound the negative term $\|x^*(z_t)-x(y_+^t(z_t),z_t)\|^2$ in the above proposition using some positive terms, so as to establish the sufficient decrease property. The following lemma explicitly quantifies the primal-dual relationship by a dual error bound.
\begin{lemma}[c.f. {\cite[Lemma D.1]{li2023nonsmooth}}]\label{xstar_yplus} Suppose Assumption \ref{f_assum} holds. For any $y\in \mathcal{Y}$ and $z\in \mathbb{R}^n$, we know
    \begin{align}\label{x_star_sub_y}
        \|x^*(z_t)-x(y_+^t(z_t),z_t)\|^2\leq \kappa \|y_t-y_+^t(z_t)\|,
    \end{align}
    where $\kappa=\frac{1+\eta_yL\sigma_2+\eta_yL}{\eta_y\left(r-L\right)}\cdot D(\mathcal{Y})$,
    and $D(\mathcal{Y})$ is the diameter of $\mathcal{Y}$.
\end{lemma}
\subsection{Convergence Theorem}
\par 
To introduce the main result, we first
define the stationarity measure that we are interested in.

\begin{definition}\label{opt_def}
	Let $\epsilon\geq 0$ be given. The point $(x,y)\in \mathcal{X}\times \mathcal{Y}$ is said to be an $\epsilon$-game-stationary point ($\epsilon$-GS) if 
	\begin{align*}
	{\rm dist}(0,\nabla_x f(x,y)+\partial \mathbf{1}_{\mathcal{X}}(x))\leq \epsilon,\\
	{\rm dist}(0, -\nabla_y f(x,y)+\partial \mathbf{1}_{\mathcal{Y}}(y))\leq \epsilon.
	\end{align*}
\end{definition}
The notion of game stationarity is a natural extension of that of first-order stationarity in a minimization problem. In the following, we will provide a sufficient condition for achieving such game stationary points.

\begin{lemma}
    Let $\epsilon\geq 0$ be given. Suppose that
    \begin{align*}
	\max\bigg\{&\frac{\|x_t-x_{t+1}\|}{\eta_x},\frac{\|y_t-y_+^t(z_t)\|}{\eta_y},\frac{\|z_{t+1}-z_t\|}{\rho},\\
	&\|\nabla_x K_t-v_t\|,\|\nabla_y K_t-w_t\|\bigg\}\leq \epsilon.
	\end{align*}
    Then, there exists a $\beta>0$ such that $(x_{t+1},y_{t+1})$ is a $\beta \epsilon$-GS.
\end{lemma}

Armed with this lemma and Proposition \ref{des_phi}, we are ready to establish the convergence results for Algorithm \ref{vr_agda}.
\begin{theorem}\label{con_theo}
    Under the setting of Proposition \ref{des_phi} and $\rho=\mathcal{O}(T^{-\frac{1}{2}})$, for any $T>0$, there exists a $t\in \{1,\ldots,T\}$ such that $(x^{t+1},y^{t+1})$ is an $\mathcal{O}(T^{-\frac{1}{4}})$-GS in expectation, i.e., 
    \begin{align}\label{theo_result}
        \mathbb{E}[{\rm dist}(0,\nabla_x f(x_{t+1},y_{t+1})+\partial \mathbf{1}_{\mathcal{X}}(x_{t+1}))]\leq \epsilon,\notag\\
        \mathbb{E}[{\rm dist}(0, -\nabla_y f(x_{t+1},y_{t+1})+\partial \mathbf{1}_{\mathcal{Y}}(y_{t+1}))]\leq \epsilon,
    \end{align}
    where $\epsilon=\mathcal{O}(T^{-\frac{1}{4}})$.
\end{theorem} 
\begin{proof}
  In the previous subsection, we have established the decrease estimate of the potential function and a dual error bound. Following the proposed analysis framework in  \cite{li2023nonsmooth}, we obtain that if $\rho=\mathcal{O}(T^{-\frac{1}{2}})$ holds, the proposed algorithm achieves an $\mathcal{O}(T^{-\frac{1}{4}})$-GS result.

    
\end{proof}

Based on the convergence result in Theorem \ref{con_theo}, we find that the proposed stochastic algorithm requires an iteration complexity of $\mathcal{O}(\epsilon^{-4})$ to achieve an $\epsilon$-GS. It matches the best-known iteration complexity of deterministic algorithms for NC-C minimax optimization \cite{zhang2020single, NEURIPS2023_a961dea4, li2023nonsmooth}, while requiring only probabilistic full gradient evaluations via the parameter $p$. 
Nevertheless, in many large-scale problems, the full gradient computation is often unfeasible due to privacy concerns or computational intractability. Motivated by this, the next section will explore a full-gradient-free extension that uses refined variance reduction techniques with auxiliary gradient trackers, achieving comparable convergence without full gradient evaluations. 

\section{Variance-Reduced Smoothed Gradient Descent-Ascent without Full Gradient}
\par The gradient estimators $v$ and $w$ in Algorithm \ref{vr_agda} require the full gradient computation with a probability of $p$, which may not be computationally efficient for solving problems with large samples. To handle this issue, we explore another variance reduction technique without full gradient calculation, called ZeroSARAH, which is firstly proposed in \cite{Li2021ZeroSARAHEN} for minimization problems. In ZeroSARAH, it develops some additional auxiliary variables for better estimation of the full gradient. To solve general NC-C minimax problem \eqref{opti_pro}, we utilize the regularized function \eqref{K_func} and propose a smoothed gradient descent-ascent algorithm with the ZeroSARAH variance reduction technique. For simplicity, we denote $\nabla K_{i,t}:=\nabla K_i(x_t,z_t;y_t)$. The gradient estimator $v_t$ and $w_t$ take the following updates
\begin{align}
    v_{t}=&\frac{1}{b}\sum_{i\in B_t} \left(\nabla_x K_{i,t}-\nabla_x K_{i,t-1}\right)+(1-\lambda)v_{t-1}\notag\\
    &+\lambda\bigg(\frac{1}{b}\sum_{i\in B_t}(\nabla_x K_{i,t-1}-d_{i,t-1})+\frac{1}{n}\sum_{i=1}^n d_{i,t-1}\bigg),\label{v_up_rul}\\
    w_{t}=&\frac{1}{b}\sum_{i\in B_t} \left(\nabla_y K_{i,t}-\nabla_y K_{i,t-1}\right)+(1-\lambda)w_{t-1}\notag\\
    &+\lambda\bigg(\frac{1}{b}\sum_{i\in B_t}(\nabla_y K_{i,t-1}-h_{i,t-1})+\frac{1}{n}\sum_{i=1}^n h_{i,t-1}\bigg),\label{w_up_rul}
\end{align}
where $B_t$ denotes a randomly sampled batch of size $b$ at iteration $t$, $d_{i,t}$ and $h_{i,t}$ are the auxiliary variables. The proposed ZeroSARAH-SGDA algorithm is summarized in Algorithm \ref{vr_zesarah}.
\begin{algorithm}
\caption{ ZeroSARAH-based Smoothed Gradient Descent-Ascent (ZeroSARAH-SGDA) algorithm}
	\label{vr_zesarah}
	\begin{algorithmic} 
        \State (S.1) : Initialize: $(x_0,y_0,z_0)$, step sizes $\eta_x>0$, $\eta_y>0$, $\rho>0$, $\lambda>0$, batch-size $b>0$, number of epochs $T$ 
        \State $x_{-1}=x_0$, $y_{-1}=y_0$, $z_{-1}=z_0$, 
        \State $v_{-1}=0$, $w_{-1}=0$, $d_{i,-1}=0$, $h_{i,-1}=0$ for all $i \in [n]$ 
        \State (S.2) Iterations:
        \For {$t=0,\cdots,T-1$}
        \State Randomly select a set of samples $B_t$ with the batch-size $|B_t|=b$
        \State Update the stochastic gradient estimators $v_t$ and $w_t$ by \eqref{v_up_rul} and \eqref{w_up_rul}, respectively
        \State $x_{t+1}=P_{\mathcal{X}}(x_t-\eta_x v_t)$
        \State $y_{t+1}=P_{\mathcal{Y}}(y_t+\eta_y w_t)$
        \State $z_{t+1}=z_{t}+\rho(x_{t+1}-z_t)$
        \State $d_{i,t}= \begin{cases} \nabla_x K_{i,t} & \text { for } i\in B_t \\ d_{i,t-1} & \text { for } i\notin B_t
        \end{cases}$
        \State $h_{i,t}= \begin{cases} \nabla_y K_{i,t} & \text { for } i\in B_t \\ h_{i,t-1} & \text { for } i\notin B_t
        \end{cases}$
        \EndFor
	\end{algorithmic}
\end{algorithm}

\begin{remark}
    Using the regularized function $K$, this paper extends the ZeroSARAH technique in minimization problems to general nonconvex-concave minimax problems, avoiding the stochastic gradient variance. Compared with the PVR-SGDA algorithm in Algorithm \ref{vr_agda}, the ZeroSARAH-SGDA does not require any full gradient computation, albeit by introducing two auxiliary gradient trackers, $d_{i,t}$ and $h_{i,t}$.  Although the ZeroSARAH-SGDA is more intricate than PVR-SGDA and demands additional storage space for auxiliary variables, we will show that ZeroSARAH-SGDA can achieve a comparable iteration complexity to PVR-SGDA, while reducing the number of gradient oracle calls per iteration.
\end{remark}
\subsection{Convergence analysis}
To analyze the proposed ZeroSARAH-SGDA algorithm, we develop a modified potential function as follows
\begin{align}\label{phi_func_def2}
    \Phi_t&=\ V_t+\gamma (\|\nabla_x K_t-v_t\|^2+\|\nabla_y K_t-w_t\|^2)\notag\\
&+\tau\left(\!\frac{1}{n}\sum_{i=1}^n\|\nabla_x K_{i,t}\!-d_{i,t}\|^2+\frac{1}{n}\sum_{i=1}^n \|\nabla_y K_{i,t}\!-h_{i,t}\|^2\!\right),
\end{align}
where $V_t\triangleq K_t-2d(y_t,z_t)+2P(z_t)$, $\gamma>0$ and $\tau>0$ are constant parameters. The modified potential function is lower bounded.
\par Similar to the theoretical analysis in Section \ref{proof_sec}, we quantify the change in each term of the potential function $\Phi_t$ after one round of updates in ZeroSARAH-SGDA. The main differences from the analysis of PVR-SGDA are the following descent properties of gradient estimators, whose proofs are provided in Appendix \ref{des_v_nab_proof}.
\begin{lemma} \label{V2_nabla_lem}
Suppose Assumption \ref{f_assum} holds. The stochastic gradient estimators $v_t$ and $w_t$ generated by Algorithm \ref{vr_zesarah} satisfy
\begin{align}\label{nabx_v_diff}
    &\mathbb{E}[\|\nabla_x K_t-v_t\|^2]-\mathbb{E}[\|\nabla_x K_{t+1}-v_{t+1}\|^2]\notag\\
    \geq \ & \lambda  \mathbb{E}[\|v_{t}-\nabla_x K_t\|^2]-\frac{2\lambda^2}{b}\mathbb{E}\left[\frac{1}{n}\sum_{i=1}^n \|\nabla_x K_{i,t}-d_{i,t}\|^2\right]\notag\\
    &-\frac{6(L+r)^2}{b}\mathbb{E}[\|x_{t+1}-x_{t}\|^2]-\frac{6L^2}{b}\mathbb{E}[\|y_{t+1}-y_{t}\|^2]\notag\\
    &-\frac{6r^2}{b}\mathbb{E}[\|z_{t+1}-z_{t}\|^2],
\end{align}
and 
\begin{align}\label{naby_w_diff}
    &\mathbb{E}[\|\nabla_y K_t-w_t\|^2]-\mathbb{E}[\|\nabla_y K_{t+1}-w_{t+1}\|^2]\notag\\
    \geq \ & \lambda  \mathbb{E}[\|w_{t}-\nabla_y K_t\|^2]-\frac{2\lambda^2}{b}\mathbb{E}\left[\frac{1}{n}\sum_{i=1}^n \|\nabla_y K_{i,t}-h_{i,t}\|^2\right]\notag\\
    &-\frac{4L^2}{b}\mathbb{E}[\|x_{t+1}-x_{t}\|^2]-\frac{4L^2}{b}\mathbb{E}[\|y_{t+1}-y_{t}\|^2],
\end{align}
where $0<\lambda<1$.
\end{lemma}
Then, the following lemma provides the descent properties of the terms $\mathbb{E}\left[\frac{1}{n}\sum_{i=1}^n\|\nabla_x K_{i,t}-d_{i,t}\|^2\right]$ and $\mathbb{E}\left[\frac{1}{n}\sum_{i=1}^n\|\nabla_y K_{i,t}-h_{i,t}\|^2\right]$ in the upper bounds of \eqref{nabx_v_diff} and \eqref{naby_w_diff}.
\begin{lemma}\label{nab_d_lem}
Suppose Assumption \ref{f_assum} holds. We have that for any $\beta>0$,
\begin{align}\label{nabx_d_diff}
    &\mathbb{E}\!\left[\!\frac{1}{n}\sum_{i=1}^n\|\nabla_x K_{i,t}
    \!-d_{i,t}\|^2\right]\!-\!\mathbb{E}\!\left[\!\frac{1}{n}\sum_{i=1}^n\|\nabla_x K_{i,t+1}\!-d_{i,t+1}\|^2\!\right]\notag\\
    &\geq \  \left(1-\zeta\right)\mathbb{E}\left[\frac{1}{n}\sum_{i=1}^n \|\nabla_x K_{i,t}-d_{i,t}\|^2\right]\notag\\
    &\qquad -3\xi(L+r)^2 \mathbb{E}\left[\|x_{t+1}\!-x_{t}\|^2\right]-3\xi L^2\mathbb{E}\left[\|y_{t+1}-y_{t}\|^2\right]\notag\\
    &\qquad -3\xi r^2 \mathbb{E}\left[\|z_{t+1}-z_{t}\|^2\right],
\end{align}
and
\begin{align}\label{naby_h_diff}
    &\mathbb{E}\!\left[\!\frac{1}{n}\sum_{i=1}^n\|\nabla_y K_{i,t}\!-h_{i,t}\|^2\!\right]\!-\!\mathbb{E}\!\left[\!\frac{1}{n}\sum_{i=1}^n\|\nabla_y K_{i,t+1}\!-h_{i,t+1}\|^2\!\right]\notag\\
    \geq & \left(1-\zeta\right)\mathbb{E}\left[\frac{1}{n}\sum_{i=1}^n \|\nabla_y K_{i,t}-\!h_{i,t}\|^2\right]\notag\\
    &-2 \xi L^2\mathbb{E}\left[\|x_{t+1}-x_{t}\|^2+\|y_{t+1}-y_{t}\|^2\right],
\end{align}
where $\zeta=\left(1-\frac{b}{n}\right)(1+\beta)$ and $\xi=\left(1-\frac{b}{n}\right)\left(1+\frac{1}{\beta}\right)$.
\end{lemma}
\par Using Lemmas \ref{V2_nabla_lem}, \ref{nab_d_lem} and \ref{V_ineq_lem}, the descent property of potential function $\Phi_t$ is shown in the following proposition, whose proof is provided in Appendix \ref{Phi_SARAH_proof}.
\begin{proposition}\label{des_phi2}
Suppose Assumption \ref{f_assum} holds.  Let $2L \leq r\leq 4L$, $\gamma = \frac{2}{\lambda}+\frac{2}{5\lambda L}$, and $\tau=2\gamma \lambda^2$. The step-sizes satisfy $\lambda=\frac{1}{b}$, $b= a\sqrt{n}$, $a\geq 2$,
\begin{align*}
    \eta_x\leq \ & \frac{b}{b(1+24L+2L^2)+310L^2\gamma+160b\tau L^2 b_+},\\
    \eta_y \leq \ &\min \left\{\frac{b}{b\!\left(2+18L\!+20\tau L^2 b_+ \!\right)\!+40\gamma L^2},\frac{1}{4L(1+\!\omega)^2}\right\},\\
    \rho \leq \ & \frac{4b}{1200b+36r \gamma+18b\tau r b_+},
\end{align*}
where $b_+=1+b$.
Then, for any $t\geq 0$, we have
    \begin{align}
    &\mathbb{E}[\Phi_t-\Phi_{t+1}]\notag\\
    \geq & c_x \mathbb{E} [\|x_{t+1}\!-x_t\|^2]\!+\!c_y \mathbb{E}[\|y_t\!-y_+^t(z_t)\|^2]\!+\!c_z\mathbb{E} [\|z_{t+1}\!-z_t\|^2] \notag\\
        &+c_v\mathbb{E}[\|\nabla_x K_t-v_t\|^2]+c_w\mathbb{E}[\|\nabla_y K_t-w_t\|^2]\notag\\
        &+c_\tau \mathbb{E}\!\left[\frac{1}{n}\!\sum_{i=1}^n\|\nabla_x K_{i,t}\!-d_{i,t}\|^2\!+\!\frac{1}{n}\!\sum_{i=1}^n \|\nabla_y K_{i,t}-h_{i,t}\|^2\right]\notag\\
        &-24r\rho \kappa \mathbb{E}[\|y_t-y_+^t(z_t)\|],
\end{align}
where $c_x=\frac{1}{2\eta_x}$, $c_y=\frac{1}{4\eta_y}$, $c_z=\frac{r}{6\rho}$, $c_v=c_w=\frac{\gamma \lambda}{2}$, $c_{\tau}=\frac{\tau}{\sqrt{n}}$.
\end{proposition}
\par Using Proposition \ref{des_phi2}, we establish the main theorem concerning the iteration complexity of Algorithm \ref{vr_zesarah} with respect to the above-mentioned standard stationary measure (GS in Definition \ref{opt_def}) for \eqref{opti_pro}.
\begin{theorem}\label{con_theo_zeroSARAH}
    Under the setting of Proposition \ref{des_phi2}, for any $T>0$, there exists a $t\in \{1,\cdots,T\}$ such that $(x^{t+1},y^{t+1})$ is an $\mathcal{O}(T^{-\frac{1}{4}})$-GS in expectation if $\rho=\mathcal{O}(T^{-\frac{1}{2}})$. In addition, the gradient complexity of Algorithm \ref{vr_zesarah} is $\mathcal{O}(\sqrt{n}\epsilon^{-4})$.
\end{theorem}
\begin{proof}
    Following a similar proof to that of Theorem \ref{con_theo}, we obtain that the proposed Algorithm \ref{vr_zesarah} achieves an $\epsilon$-GS with an iteration complexity of $\mathcal{O}(\epsilon^{-4})$. In addition, the proposed algorithm requires $a\sqrt{n}$ ($a\geq 2$) gradient computations at each iteration, hence, the gradient complexity is $\mathcal{O}(\sqrt{n}\epsilon^{-4})$. 
\end{proof}
\begin{remark}\label{comp_rem}
    Based on Theorem \ref{con_theo_zeroSARAH}, the ZeroSARAH-SGDA algorithm achieves the same iteration complexity as that of the PVR-SGDA algorithm, while reducing the number of gradient oracle calls at each iteration. 
    Although the ZeroSARAH-SGDA algorithm exhibits a lower gradient complexity than the PVR-SGDA algorithm  by eliminating the need for full gradient computations, it adopts a more intricate design and necessitates additional storage space for maintaining and updating the auxiliary gradient trackers $d_{i,t}$ and $h_{i,t}$. Consequently, a fundamental trade-off emerges between Algorithm \ref{vr_agda} and Algorithm \ref{vr_zesarah}, balancing computational efficiency against memory resource demands.
\end{remark}

\section{Numerical Results}\label{sec:simulation}
This section applies the proposed PVR-SGDA and ZeroSARAH-SGDA algorithms to the robust logistic regression and data poisoning to demonstrate their practical efficacy. The logistic regression problem acts as a standard test ground for various algorithms in robust learning, such as \cite{zhang2022sapd,den_nonmini_nips}. Data Poisoning is an adversarial attack that the attacker tries to manipulate the training dataset. In both applications, we compare the proposed two algorithms with several existing algorithms, including the popular StocGDA algorithm \cite{SGDA} and the SVRG-based variance-reduced AGDA (VR-AGDA) algorithm \cite{NEURIPS2020_0cc6928e}. It is important to note that VR-AGDA is designed for minimax problems that satisfy the one-sided Polyak-{\L}ojasiewicz inequality, and its theoretical analysis does not directly apply to the NC-C minimax problems considered here. We only focus on the numerical performance of VR-AGDA in this context. 

\subsection{Robust Logistic Regression}
\begin{figure*}[t]
	\centering
	\subfigure{
		\includegraphics[width=8cm]{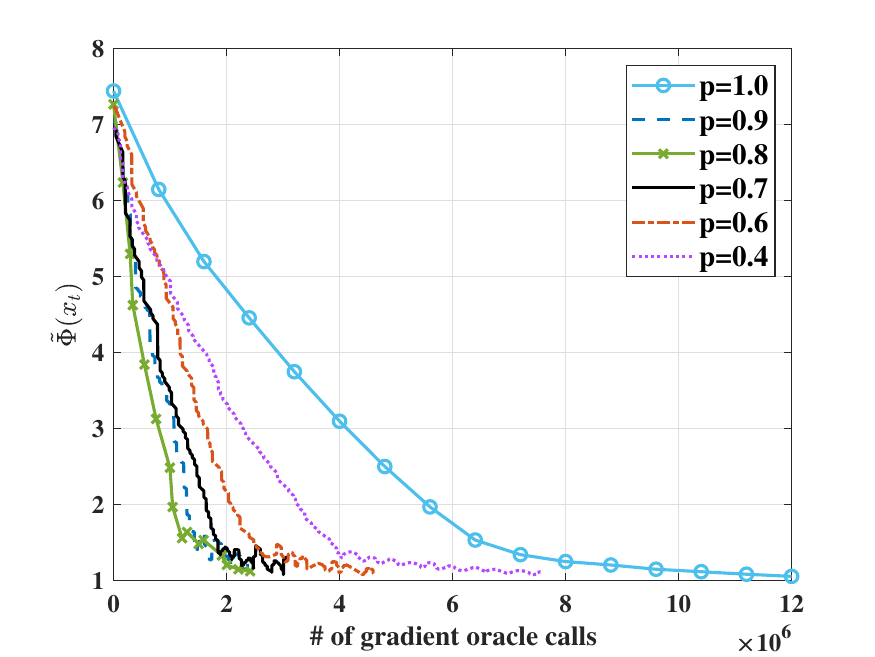}
		\label{RLR_posi_fig}
	}\qquad
	\subfigure{
		\includegraphics[width=8cm]{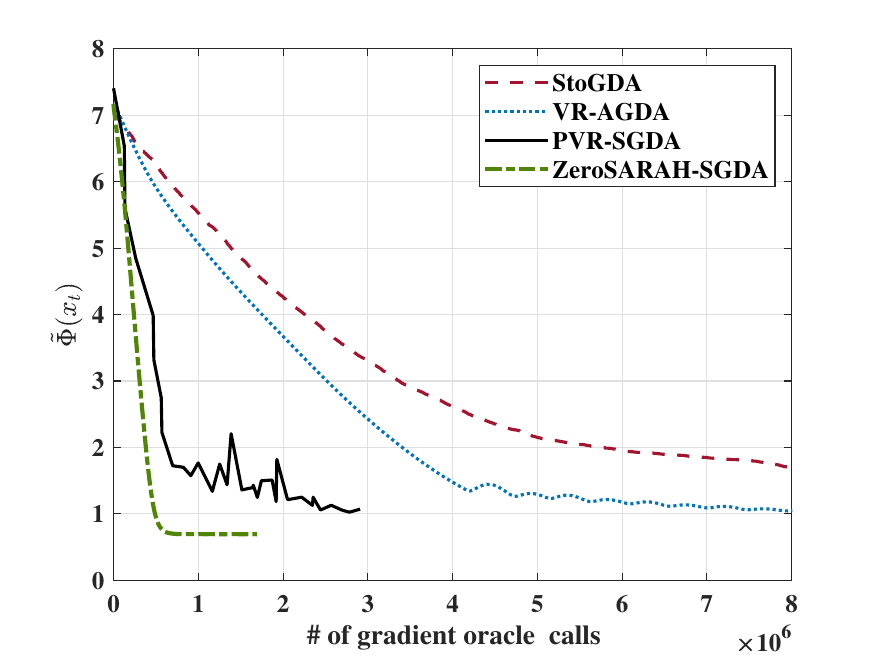}
		\label{RLR_a9a_fig}
	}
	\caption{(a) Convergence of  PVR-SGDA algorithm with different $p$. (b) Performance for different algorithms.}
	\label{RLR_fig}
\end{figure*}
\begin{figure}
	\centering
		\includegraphics[width=8cm]{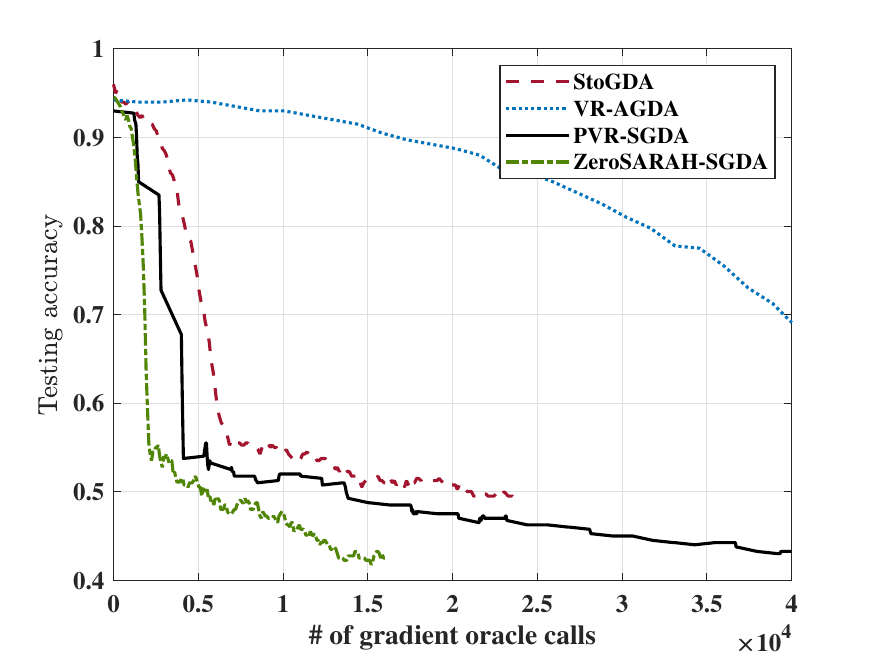}
	\caption{Testing accuracy with respect to gradient oracle calls in data poisoning.}
 \label{poison_fig}
\end{figure}
 For a public dataset $\left\{\left(a_i, b_i\right)\right\}_{i=1}^n$, where $a_i \in \mathbb{R}^d$ is the feature vector and $b_i \in\{-1,1\}$ is the label, the nonconvex-regularized problem  is formulated as follows:
$$
\min _{x \in \mathbb{R}^d} \max _{y \in \Delta_n} f(x, y)=\sum_{i=1}^n y_i \log \left(1+\exp \left(-b_i a_i^\top x\right)\right)+g(x),
$$
where $y_i$ is the $i$-th component of variable $y$ and $\Delta_n$ denotes the simplex in $\mathbb{R}^n$. The nonconvex regularization $g$ has the form $g(x):=\lambda_2 \sum_{i=1}^d \alpha x_i^2/(1+\alpha x_i^2)$. 
Following the settings in \cite{SREDA_nips2020,den_nonmini_nips}, we set $\lambda_1=1/n^2$, $\lambda_2=0.001$, and $\alpha=10$ in our experiment.
We conduct the experiment on the public dataset a9a, where $d=123$ and $n=32561$. To measure the convergence performance of algorithms, we evaluate the function value $\tilde{\Phi}(x)=\max_{y\in \Delta_n} f(x,y)$ with respect to the number of gradient oracles. 

\par Figure \ref{RLR_posi_fig} illustrates the convergence trajectories of the PVR-SGDA algorithm for different values of the probability $p$. It can be observed that as $p$ increases, the convergence rate initially accelerates but then decreases as $p$ approaches 1. When $p = 1$, the proposed algorithm reduces to the deterministic smoothed GDA algorithm as described in \cite{zhang2020single}. When $p$ increases to a certain value, such as $p=0.8$ in Figure \ref{RLR_posi_fig}, the convergence trajectory slows down due to the increased gradient computation burden at each iteration. Thus, a trade-off exists between the choice of $p$ and the overall convergence efficiency.  
\par In addition, we compare the performance of the proposed PVR-SGDA and ZeroSARAH-SGDA algorithms with the existing StocGDA and VR-AGDA algorithms. The convergence trajectories with respect to the number of gradient oracle calls are presented in Figure \ref{RLR_a9a_fig}. From these trajectories, we observe that the proposed PVR-SGDA and ZeroSARAH-SGDA algorithms converge faster than the other baseline algorithms, thereby validating the effectiveness of the proposed algorithms. In addition, the proposed ZeroSARAH-SGDA algorithm converges faster with respect to the number of gradient oracle calls than that of the PVR-SGDA algorithm, which verifies the lower gradient complexity of the ZeroSARAH-SGDA algorithm.
\subsection{Data poisoning}
\par The goal of the attacker is to corrupt the training dataset so that predictions on the dataset will be modified in the testing phase when using a machine learning model. Let $\mathcal{D} = \{z_i, t_i\}_{i=1}^n$ denote the training dataset, where $n'\ll n$ samples are corrupted by a perturbation vector $x$, resulting in poisoned training data $z_i + x$. These perturbed samples disrupt the training process and reduce prediction accuracy. Following the setup in \cite{xu2024derivativefree,pmlr-v119-liu20j}, we generate a dataset containing $n = 1000$ samples $\{z_i, t_i\}_{i=1}^n$, where $z_i \in \mathbb{R}^{100}$ are sampled from $\mathcal{N}(\mathbf{0}, \mathbf{I})$. With $\nu_i \sim \mathcal{N}(0, 10^{-3})$, we set
\[
t_i = \begin{cases} 
1, & \text{if} \; 1/(1+e^{-z_i^T \theta^*+\nu_i}) > 0.5, \\
0, & \text{otherwise}.
\end{cases}
\]
\par Here, we choose $\theta^*$ as the base model parameters. The dataset is randomly split into a training dataset $\mathcal{D}_{\text{train}}$ and a testing dataset $\mathcal{D}_{\text{test}}$. The training dataset is further divided into $\mathcal{D}_{\operatorname{tr}, 1}$ and $\mathcal{D}_{\operatorname{tr}, 2}$, where $\mathcal{D}_{\operatorname{tr}, 1}$ (resp. $\mathcal{D}_{\operatorname{tr}, 2}$) represents the poisoned (resp. unpoisoned) subset of the training dataset. The problem can then be formulated as follows:
\begin{align*}
    \underset{\|x\|_{\infty} \leq \epsilon}{\max} \underset{\theta}{\min}\, f(x, \theta;\mathcal{D}_{\text{train}}),
\end{align*}
where $f(x, \theta ; \mathcal{D}_{\text{train}}):=F(x, \theta ; \mathcal{D}_{\operatorname{tr}, 1})+F(0, \theta ; \mathcal{D}_{\operatorname{tr}, 2})$ with 
\begin{align*}
F(x, \theta;\mathcal{D})
=&-\frac{1}{|\mathcal{D}|} \sum_{(z_i,t_i) \in \mathcal{D}}[t_i \log (l(x,\theta;z_i))\\
&+(1-t_i) \log (1-l(x, \theta;z_i))],
\end{align*}
and $l(x, \theta ; z_i)=1/(1+e^{-(z_i+x)^T \theta})$.
\begin{table}
\centering
\caption{Result Comparisons}
\label{mem_table}
\begin{tabular}{|c|c|c|c|}
\hline
Algorithm & space complexity & time (s)  & Testing accuracy \\ \hline
StoGDA     & $\mathcal{O}(1)$    &  5.29        & 0.49                \\ \hline
VR-AGDA     & $\mathcal{O}(n)$   & 9.87          & 0.45                \\ \hline
PVR-SGDA      & $\mathcal{O}(1)$    & 3.2         & 0.38                \\ \hline
ZeroSARAH-SGDA      & $\mathcal{O}(n)$  & 3.08           & 0.4                \\ \hline
\end{tabular}
\end{table}
\par We set the poisoning ratio $|\mathcal{D}_{\text{tr},1}|/|\mathcal{D}_{\text{train}}| = 10\%$, $\epsilon = 2$, and the number of iterations $T = 200$. To compare the performance of different algorithms, including StoGDA, VR-AGDA, the proposed PVR-SGDA and ZeroSARAH-SGDA algorithms, we evaluate the prediction accuracy of the generated models on the testing dataset. 
Fig. \ref{poison_fig} presents the trajectories of prediction accuracy with respect to the number of gradient oracle calls. It is observed that the testing accuracies of the generated models by the proposed PVR-SGDA and ZeroSARAH-SGDA algorithms decrease faster than those of the other StoGDA and VR-AGDA baseline algorithms. In addition, the testing accuracy trajectory of the ZeroSARAH-SGDA algorithm decreases faster than that of PVR-SGDA, which verifies the low gradient complexity of the ZeroSARAH-SGDA algorithm. 
\par Table \ref{mem_table} summarizes the storage space complexity, execution time, and the final prediction accuracy of different comparable algorithms. Notably, the PVR-SGDA algorithm exhibits significantly lower storage demands compared to ZeroSARAH-SGDA, making it the preferable choice in memory-constrained scenarios where gradient computational cost is less of a concern. We note that the computational cost of PVR-SGDA algorithm is closely related to the probability $p$ of full gradient calculation. Finally, the proposed PVR-SGDA and ZeroSARAH-SGDA algorithms demonstrate lower prediction accuracies on the testing dataset compared to  other baseline algorithms, indicating that the proposed algorithms achieve superior attack performance.

\section{Conclusion}
\label{sec:conclusion}

This paper presents two single-loop variance-reduced stochastic gradient algorithms for solving NC-C minimax problems. By incorporating a probability-based variance reduction step, the proposed PVR-SGDA algorithm reduces the number of full gradient calculations, which is typical in deterministic algorithms, and achieves convergence that is robust to gradient variance. Moreover, to completely eliminate the need for full gradient computations, this paper introduces the refined variance reduction step using auxiliary gradient trackers and develops the novel ZeroSARAH-SGDA algorithm with a low gradient complexity. Utilizing the Moreau-Yosida smoothing technique, the two proposed algorithms achieve the best-known complexity of $\mathcal{O}(\epsilon^{-4})$ for NC-C minimax problem, surpassing the $\mathcal{O}(\epsilon^{-6})$ complexity of existing stochastic algorithms. In numerical results, the ZeroSARAH-SGDA algorithm demonstrates better convergence performance than the PVR-SGDA algorithm regarding the number of gradient oracle calls.

\begin{appendices}
\section{Useful Lemmas}
\begin{lemma}\label{x_lem}
For any $y_1,y_2\in \mathcal{Y}$ and $z_1,z_2\in \mathbb{R}^n$,
    there exist constants $\sigma_1,\sigma_2,\sigma_3$ independent of $y$ such that
    \begin{align}
        \|x(y,z_1)-x(y,z_2)\|\leq \ &\sigma_1 \|z_1-z_2\|, \label{z_sub_sigma}\\
        \|x^*(z_1)-x^*(z_2)\|\leq \ &\sigma_1 \|z_1-z_2\|, \label{xstar_z_sub_sigma}\\
        \|x(y_1,z)-x(y_2,z)\|\leq \ &\sigma_2 \|y_1-y_2\|,\label{y_sub_sigma}\\
        \|x_{t+1}-x(y_t,z_t)\|^2 \leq \ &2\eta_x^2\|\nabla_x K_t-v_t\|^2\notag\\
        &+2(1+\omega)^2\|x_t-x_+(y_t,z_t)\|^2, \label{x_plus_sigma}
    \end{align}
    where $\sigma_1\coloneqq \frac{r}{r-L}$, $\sigma_2\coloneqq \frac{2r-L}{r-L}$ and $\omega\coloneqq \frac{\eta_x L+\eta_x r+1}{\eta_x r-\eta_x L}$.
\end{lemma}
\begin{proof}
    The \eqref{z_sub_sigma}, \eqref{xstar_z_sub_sigma} and \eqref{y_sub_sigma} follow from \cite[Lemma A.1]{li2023nonsmooth}. Then, we prove the inequality \eqref{x_plus_sigma}. 
     Lemma \ref{K_sc_m} shows that the mapping $\nabla_x K(\cdot,z;y)$ is $\left(r-L\right)$-strongly monotone and Lipschitz continuous with constant $\left(r+L\right)$ on the set $\mathcal{X}$. Adopting the proof in  \cite[Theorem 3.1]{Pang_VIp}, we have
\begin{align}\label{ome_x}
   \left\|x_t-x\left(y_t, z_t\right)\right\| \leq \frac{\eta_x L+\eta_x r+1}{\eta_x r-\eta_x L}\left\|x_t-x_+(y_t,z_t)\right\|. 
\end{align}
    Using the triangle inequality yields $\|x(y_t,z_t)-x_+(y_t,z_t)\|\leq \|x(y_t,z_t)-x_t\|+\|x_t-x_+(y_t,z_t)\|\leq \left(1+\omega\right)\|x_t-x_+(y_t,z_t)\|$, where $\omega\coloneqq \frac{\eta_x L+\eta_x r+1}{\eta_x r-\eta_x L}$. In addition, with the non-expansive property of projection operator, we have
    \begin{align}\label{x_x_plus}
        \|x_{t+1}\!-x_+(y_t,z_t)\|=\ &\|P_{\mathcal{X}}(x_t\!-\eta_x v_t)\!-\!P_{\mathcal{X}}(x_t\!-\eta_x \nabla_x K_t)\|\notag\\
        \leq \ & \eta_x\|\nabla_x K_t-v_t\|.
    \end{align}
    Then, combining these two inequalities yields the result \eqref{x_plus_sigma}.
\end{proof}
\begin{lemma} For any $t$, the generated variables satisfy
        \begin{align}
             \|y_t-y_+^t(z_t)\|^2/2\leq\ & \|y_{t+1}-y_t\|^2+2\eta_y^2\|w_t-\nabla_y K_t\|^2\notag\\
             &+2\eta_y^2L^2
        \omega^2\|x_t-x_+(y_t,z_t)\|^2,
        \end{align}
\end{lemma}
where $\omega=\frac{\eta_x L+\eta_x r+1}{\eta_x r-\eta_x L}$.
\begin{proof}
    Using the fact that $\|a-b\|^2\geq \|b\|^2/2-\|a\|^2$, we have
    \begin{align}\label{yt_rel}
        \|y_{t+1}-y_t\|^2
        =&\|y_{t+1}-y_+^t(z_t)+y_+^t(z_t)-y_t\|^2\notag\\
        \geq&\|y_t-y_+^t(z_t)\|^2/2-\|y_{t+1}-y_+^t(z_t)\|^2 \notag\\
        \geq&\|y_t-y_+^t(z_t)\|^2/2-2\eta_y^2\|w_t-\nabla_y K_t\|^2\notag\\
        &-2\eta_y^2L^2
        \omega^2\|x_t-x_+(y_t,z_t)\|^2,
    \end{align}
    where the last equality holds due to the fact that $ \|y_{t+1}-y_+^t(z_t)\|^2\leq\ 2\eta_y^2\|w_t-\nabla_y K_t\|^2+2\eta_y^2L^2
        \omega^2\|x_t-x_+(y_t,z_t)\|^2$ by using the non-expansiveness of the projection operator and  \eqref{ome_x}. 
\end{proof}
\section{Descent of $V$}\label{V_des_proof_Sec}
\begin{lemma}[Primal Descent]
For any $t$, the regularized function $K$ satisfies
\begin{align*}
     &K_t-K_{t+1}\notag\\
     \geq& -\frac{1}{2}\|\nabla_x K_t -v_t\|^2+\left(\frac{1}{\eta_x}-\frac{r+L+1}{2}\right)\|x_{t+1}-x_t\|^2 \notag\\
        &+\langle \nabla_y K(x_{t+1},z_t;y_t),y_t-y_{t+1}\rangle-\frac{L}{2}\|y_{t}-y_{t+1}\|^2 \\
        &+\frac{r\left(2-\rho\right)}{2\rho}\|z_t-z_{t+1}\|^2.
\end{align*}
\end{lemma}
\begin{proof}
    By the updating $x_{t+1}=P_{\mathcal{X}}(x_t-\eta_x v_t)$, we have
    \begin{align}\label{pro_xv}
        v_t^\top (x_{t+1}-x_t)\leq -\frac{\|x_{t+1}-x_t\|^2}{\eta_x}.
    \end{align}
    Since $K(\cdot,z;y)$ is $(L+r)$-smooth,
    \begin{align}\label{K_x}
        &K_t-K(x_{t+1},z_t;y_t)\notag\\
        \geq & -\langle \nabla_x K_t,x_{t+1}-x_t\rangle-\frac{r+L}{2}\|x_{t+1}-x_t\|^2\notag\\
        = &-\langle \nabla_x K_t-v_t+v_t,x_{t+1}-x_t\rangle-\frac{r+L}{2}\|x_{t+1}-x_t\|^2\notag\\
        \geq &-\langle \nabla_x K_t-v_t,x_{t+1}-x_t\rangle+\frac{\|x_{t+1}-x_t\|^2}{\eta_x}\notag\\
        &-\frac{r+L}{2}\|x_{t+1}-x_t\|^2\notag\\
        \geq &-\frac{1}{2}\|\nabla_x K_t-v_t\|^2+\left(\frac{1}{\eta_x}-\frac{r+L+1}{2}\right)\|x_{t+1}-x_t\|^2,
    \end{align}
    where the second last inequality holds by \eqref{pro_xv} and the last inequality is from the AM-GM inequality.
    \par Next, because $\nabla_y K(x,z;\cdot)$ is $L$-Lipschitz continuous, we have
    \begin{align}\label{K_y}
        &K(x_{t+1},z_t;y_t)-K(x_{t+1},z_t;y_{t+1})\notag\\
        \geq \ &\langle \nabla_y K(x_{t+1},z_t;y_t),y_t-y_{t+1}\rangle-\frac{L}{2}\|y_{t}-y_{t+1}\|^2.
    \end{align}
    Based on the updating $z_{t+1}=z_t+\rho \left(x_{t+1}-z_t\right)$, it is easy to obtain
    \begin{align}\label{K_z}
        K(x_{t+1},z_t;y_{t+1})-K_{t+1}= \frac{r\left(2-\rho\right)}{2\rho}\|z_t-z_{t+1}\|^2.
    \end{align}
    Combining \eqref{K_x}, \eqref{K_y} and \eqref{K_z}, we obtain
    \begin{align*}
        &K_t-K_{t+1}\\
        \geq& -\frac{1}{2}\|\nabla_x K_t -v_t\|^2+\left(\frac{1}{\eta_x}-\frac{r+L+1}{2}\right)\|x_{t+1}-x_t\|^2 \\
        &+\langle \nabla_y K(x_{t+1},z_t;y_t),y_t-y_{t+1}\rangle-\frac{L}{2}\|y_{t}-y_{t+1}\|^2 \\
        &+\frac{r\left(2-\rho\right)}{2\rho}\|z_t-z_{t+1}\|^2.
    \end{align*}
\end{proof}

\begin{lemma}[{Dual Ascent; c.f. \cite[Lemma B.2]{li2023nonsmooth}}]
    For any $t$, the dual function $d$ satisfies
    \begin{align*}
        &d(y_{t+1},z_{t+1})-d(y_t,z_t)\\
        \geq \ &\langle \nabla_y K(x(y_t,z_t),z_t;y_t),y_{t+1}-y_t\rangle-\frac{L_d}{2}\|y_t-y_{t+1}\|^2\notag\\
        &+\frac{r}{2}\langle z_{t+1}-z_t,z_{t+1}+z_t-2x(y_{t+1},z_{t+1})\rangle,
    \end{align*}
    with $L_d=L+L\sigma_2$.
\end{lemma}
Recall that $Y(z):= \{y\in \mathcal{Y}|\ {\rm argmax}_{y\in \mathcal{Y}} \ d(y,z)\}$ and $P(z):=d(y(z),z)$, where $y(z)\in Y(z)$.
\begin{lemma}[{Proximal Descent; c.f. \cite[Lemma B.3]{li2023nonsmooth}}]
For any $t$, the proximal function $P$ satisfies
    \begin{align*}
        P(z_{t+1})\!-P(z_t)\leq \frac{r}{2}\langle z_{t+1}\!-z_t, z_t+z_{t+1}\!-\!2x(y(z_{t+1}),z_t)\rangle,
    \end{align*}
    where $y(z_{t+1})\in Y(z_{t+1})$.
\end{lemma}
\textbf{Proof of Lemma \ref{V_ineq_lem}:}
\begin{proof}
    Incorporating above lemmas, we have
    \[
        \begin{aligned}
        &V_t-V_{t+1}\\
        \geq\ & -\frac{1}{2}\|\nabla_x K_t-v_t\|^2+\left(\frac{1}{\eta_x}-\frac{r+L+1}{2}\right)\|x_{t+1}-x_t\|^2\\
        &-\left(\frac{L}{2}+L_d\right)\|y_{t}-y_{t+1}\|^2+\frac{r(2-\rho)}{2\rho}\|z_t-z_{t+1}\|^2 \notag\\
        &\!+\!\underbrace{\langle 2\nabla_y K(x(y_t,z_t),z_t;y_t)\!-\!\nabla_y K(x_{t+1}\!,z_t;y_t),y_{t+1}\!-\!y_t\rangle}_{\text{\ding{192}}} \\
        &+\underbrace{2r\langle z_{t+1}-z_t, x(y(z_{t+1}),z_t))-x(y_{t+1},z_{t+1})\rangle}_{\text{\ding{193}}}.        
        \end{aligned}
        \]
    For the first term, using the update of $y_{t+1}$, we have 
    \[
    \begin{aligned}
        \text{\ding{192}} 
        \geq\ & 2\langle \nabla_y K(x(y_t,z_t),z_t;y_t)\!-\!\nabla_y K(x_{t+1},z_t;y_t),y_{t+1}\!-\!y_t\rangle \\
        &-\frac{1}{2}\|\nabla_y K(x_{t+1},z_t;y_t)-w_t\|^2\\
        &-\frac{1}{2}\|y_{t+1}-y_t\|^2 +\frac{\|y_{t+1}-y_t\|^2}{\eta_y}\\
        \geq \ & 2\langle \nabla_y K(x(y_t,z_t),z_t;y_t)\!-\!\nabla_y K(x_{t+1},z_t;y_t),y_{t+1}\!-\! y_t\rangle \\
        &-\|\nabla_y K_t-w_t\|^2-L^2\|x_{t+1}-x_t\|^2\\
        &-\frac{1}{2}\|y_{t+1}-y_t\|^2 +\frac{\|y_{t+1}-y_t\|^2}{\eta_y}.
    \end{aligned}
    \]
    Moreover, by the Cauchy-Schwarz inequality, we have
    \begin{align*}
        &2 \langle\nabla_y K(x(y_t,z_t),z_t;y_t)-\nabla_y K(x_{t+1},z_t;y_t),y_{t+1}-y_t\rangle\\
        \geq\ & \!-\!2\|\nabla_y K(x(y_t,z_t),z_t;y_t\!)\!-\!\nabla_y K(x_{t+1}\!,z_t;y_t)\|\|y_{t+1}\!-\!y_t\|\\
        \geq\ &\!-\!2 L\|x_{t+1}-x(y_t,z_t)\|\|y_{t+1}-y_t\|\\
        \geq\ & \!-\! L (1+\omega)^2 \|y_{t+1}-y_t\|^2-\frac{2 \eta_x^2 L}{(1+\omega)^2} \|\nabla_x K_t-v_t\|^2\\
        &\!-\! 2 L\|x_t-x_+(y_t,z_t)\|^2,
    \end{align*}
    where the last inequality holds due to \eqref{x_plus_sigma} and the AM-GM inequality. Armed with this inequality, we further upper bound term \ding{192} as follows:
    \[
    \begin{aligned}
         \text{\ding{192}}  \geq & -\|\nabla_y K_t-w_t\|^2-L^2\|x_{t+1}-x_t\|^2\\
         &+\left(\frac{1}{\eta_y}-\frac{1}{2}-L(1+\omega)^2\right)\|y_{t+1}-y_t\|^2\\
         &-\frac{2 \eta_x^2 L}{(1+\omega)^2} \|\nabla_x K_t-v_t\|^2-2 L\|x_t-x_+(y_t,z_t)\|^2.
    \end{aligned}
    \]
    Next, the second term studied in \cite{zhang2020single} is bounded by  
   \[ 
   \begin{aligned}
       \text{\ding{193}}\geq & -2r\sigma_1\|z_{t+1}-z_t\|^2-\frac{r}{6\rho}\|z_{t+1}-z_t\|^2\\
       &-{6r\rho} \|x(y(z_{t+1}),z_{t+1})-x(y_{t+1},z_{t+1})\|^2. 
   \end{aligned}
   \]
  Plugging the upper bound for \ding{192} and \ding{193}, we derive that 
  \begin{equation}
      \label{V_diff_xyz}
  \begin{aligned}
      &V_t-V_{t+1}\\
      \geq\ & -\left(\frac{1}{2}+\frac{2 \eta_x^2 L}{(1+\omega)^2}\right)\|\nabla_x K_t-v_t\|^2\\
      &-\|\nabla_y K_t-w_t\|^2\\
      &+\left(\frac{1}{\eta_x}-\frac{r+L+1}{2}-L^2\right)\|x_{t+1}-x_t\|^2 \\
      &+\left(\frac{1}{\eta_y}-\frac{L+1}{2}-L(1+\omega)^2-L_d\right)\|y_{t+1}-y_t\|^2\\
      &+\left(\frac{5r}{6\rho}-\frac{r}{2}-2r\sigma_1\right)\|z_t-z_{t+1}\|^2 \\
      &-2 L\|x_t-x_+(y_t,z_t)\|^2 \\
      &-{6r\rho} \underbrace{ \|x(y(z_{t+1}),z_{t+1}))-x(y_{t+1},z_{t+1})\|^2}_{\text{\ding{194}}}.
  \end{aligned}
  \end{equation}  
  Note that $x(y(z_{t+1}),z_{t+1})=x^*(z_{t+1})$, we bound \ding{194} as follows:
    \begin{align*}
       \text{\ding{194}}=\ &\|x^*(z_{t+1})\!-\!x^*(z_t)\!+\!x^*(z_t)\!-\! x(y_+^t(z_t),z_t) \\
       &+x(y_+^t(z_t),z_t)-x(y_{t+1},z_t)\\
       &+x(y_{t+1},z_t)-x(y_{t+1},z_{t+1})\|^2\notag\\
        \leq\ &4 \sigma_1^2\|z_t-z_{t+1}\|^2+4\|x^*(z_t)-x(y_+^t(z_t),z_t)\|^2 \\
        &+4\sigma_2^2\|y_+^t(z_t)-y_{t+1}\|^2+4\sigma_1^2\|z_t-z_{t+1}\|^2\notag\\
        \leq\ & 8 \sigma_1^2\|z_t-z_{t+1}\|^2+4\|x^*(z_t)-x(y_+^t(z_t),z_t)\|^2 \\
        &+8\sigma_2^2\eta_y^2\|w_t\!-\!\nabla_y K_t\|^2 \!+\!8\sigma_2^2\eta_y^2L^2\omega^2\|x_t\!-\!x_+(y_t,z_t)\|^2,
    \end{align*}
    where the first inequality holds from Lemma \ref{x_lem} and the last inequality holds by using $ \|y_{t+1}-y_+^t(z_t)\|^2\leq\ 2\eta_y^2\|w_t-\nabla_y K_t\|^2+2\eta_y^2L^2
        \omega^2\|x_t-x_+(y_t,z_t)\|^2$. 
    Then, substituting the above pieces into \eqref{V_diff_xyz} yields that
   \[
    \begin{aligned}
          &V_t-V_{t+1}\\
          \geq\ &\left(\frac{1}{\eta_x}-\frac{r+L+1}{2}-L^2\right)\|x_{t+1}-x_t\|^2  \\
          &+\left(\frac{1}{\eta_y}-\frac{L+1}{2}-L(1+\omega)^2-L_d\right)\|y_{t+1}-y_t\|^2\\
          &+\left(\frac{5r}{6\rho}-\frac{r}{2}-2r\sigma_1-48r\rho \sigma_1^2\right)\|z_t-z_{t+1}\|^2 \\
          &-2 L\left(1+24\omega^2r\rho\sigma_2^2\eta_y^2L\right)\|x_t-x_+(y_t,z_t)\|^2 \\
        &  -\left(\frac{1}{2}+\frac{2 \eta_x^2 L}{(1+\omega)^2} \right) \|\nabla_x K_t-v_t\|^2\\
        &-\left(1+48r\rho \eta_y^2\sigma_2^2\right)\|\nabla_y K_t-w_t\|^2\\
        &-24r\rho \|x^*(z_t)-x(y_+^t(z_t),z_t)\|^2.
    \end{aligned}   
   \]
\end{proof}
\section{DESCENT OF GRADIENT ESTIMATORS}\label{gradient_des_appendix}
\begin{proof}
     Following the proof sketch of  \cite[Lemma 3]{page_p}, we provide the bound for the stochastic gradient estimators. 
    The updating of variable $v$ in Algorithm \ref{vr_agda} reveals that
    \begin{align*}
        &\mathbb{E}_t[\|v_{t+1}-\nabla_x K_{t+1}\|^2]\\
        =\ &\left(1-p\right) \mathbb{E}_t\left[\|v_{t}-\nabla_x K_t+\nabla_x \tilde{K}_{t+1}-\nabla_x\tilde{K}_t\right.\\
        &\left.\qquad -\nabla_x K_{t+1}+\nabla_x K_t\|^2\right]\\
        =\ &\left(1-p\right) \mathbb{E}_t\left[\|\nabla_x \tilde{K}_{t+1}-\nabla_x\tilde{K}_t-\nabla_x K_{t+1}+\nabla_x K_t\|^2\right]\\
        &+(1-p) \|v_t-\nabla_x K_t\|^2\\
        \leq\ &3 \left(1-p\right) \mathbb{E}_t[(L+r)^2\|x_{t+1}-x_t\|^2+L^2\|y_{t+1}-y_t\|^2\\
        &+r^2\|z_{t+1}-z_t\|^2]+\left(1-p\right) \|v_t-\nabla_x K_t\|^2,
    \end{align*}
     where the last inequality holds due to the smoothness of function $K$ and the fact that $\mathbb{E}\|x-\mathbb{E}x\|^2\leq \mathbb{E}\|x\|^2$ for any random variable $x$. Hence, by taking  expectation on both sides and rearrangement, the inequality \eqref{nablax_v} holds. The inequality \eqref{nablay_w} holds by a similar analysis. 
\end{proof}
\section{Descent of $\Phi$ in PVR-SGDA}\label{des_phi_pvr}
\begin{proof}
\par With Lemma \ref{V_ineq_lem}, we bound the difference of $\Phi_t$ between two consecutive steps,
\begin{align*}
    \label{big_ineq}
        &\mathbb{E} [\Phi_t-\Phi_{t+1}]\\
        \geq\ &s_x \mathbb{E}[\|x_{t+1}\!-x_t\|^2]  + s_y\mathbb{E}[\|y_{t+1}\!-y_t\|^2] + s_z\mathbb{E}[\|z_t\!-z_{t+1}\|^2] \\
        &-2 L\left(1+24\omega^2r\rho\sigma_2^2\eta_y^2L\right)\mathbb{E}[\|x_t-x_+(y_t,z_t)\|^2] \\
        &-24r\rho \mathbb{E}[\|x^*(z_t)-x(y_+^t(z_t),z_t)\|^2]\\
        &+\!\left(\!\frac{\gamma}{2p}\!-\!s_v\!\right)\mathbb{E}[\|\nabla_x K_t\!-\!v_t\|^2] \!+\!\left(\!\frac{\gamma}{2p}\!-\!s_w\!\right)\mathbb{E}[\|\nabla_y K_t\!-\!w_t\|^2]\\
        &-\frac{\gamma}{2p}\mathbb{E}[\|\nabla_x K_{t+1}-v_{t+1}\|^2]-\frac{\gamma}{2p}\mathbb{E}[\|\nabla_y K_{t+1}-w_{t+1}\|^2]\\
        \geq\ & \left(s_x-\frac{\left(1-p\right)\gamma \left(3\left(L+r\right)^2+2L^2\right)}{2p}\right) \mathbb{E}[\|x_{t+1}-x_t\|^2]\\
        &+ \left(s_y-\frac{5\left(1-p\right)\gamma L^2}{2p} \right)\mathbb{E}[\|y_{t+1}-y_t\|^2] \\
        &+ \left(s_z-\frac{3\left(1-p\right)\gamma r^2}{2p}\right)\mathbb{E}[\|z_t-z_{t+1}\|^2] \\
        &-2 L\left(1+24\omega^2r\rho\sigma_2^2\eta_y^2L\right)\mathbb{E}[\|x_t-x_+(y_t,z_t)\|^2] \\
        &-24r\rho \mathbb{E}[\|x^*(z_t)-x(y_+^t(z_t),z_t)\|^2]\\
        &+\left( \frac{\gamma}{2}\!-\!s_v\!\right)\mathbb{E}[\|\nabla_x K_t\!-\!v_t\|^2]\!+\!\left( \frac{\gamma}{2}\!-\!s_w\!\right)\mathbb{E}[\|\nabla_y K_t\!-\!w_t\|^2],
\end{align*}
where the second inequality is due to Lemma \ref{lemma:key}.
To make the terms consistent, we further use the inequalities  \eqref{yt_rel} to bound it as follows:
\[
\begin{aligned}
        &\mathbb{E} [\Phi_t-\Phi_{t+1}]\\
        \geq\ &\left(s_x-\frac{\left(1-p\right)\gamma \left(3\left(L+r\right)^2+2L^2\right)}{2p}\right) \mathbb{E}[\|x_{t+1}-x_t\|^2]\\
        &+ \frac{1}{2}\left(s_y-\frac{5(1-p)\gamma L^2}{2p} \right)\mathbb{E}[\|y_+^t(z_t)-y_t\|^2] \\
        &+ \left(s_z-\frac{3\left(1-p\right)\gamma r^2}{2p}\right)\mathbb{E}[\|z_t-z_{t+1}\|^2] \\
        &-2L\left(1+ 24\omega^2r\rho \sigma_2^2 \eta_y^2 L+\eta_y^2L\omega^2(s_y\!-\!\frac{5(1\!-p)\gamma L^2}{2p})\right)\\
        &\qquad \mathbb{E}[\|x_t-x_+(y_t,z_t)\|^2]\\
        &+\left( \frac{\gamma}{2}-s_v\right)\mathbb{E}[\|\nabla_x K_t-v_t\|^2]\\
        &+\left( \frac{\gamma}{2}\!-\!s_w\!-\!2\eta^2_y \left(s_y\!-\!\frac{5\left(1\!-\!p\right)\gamma L^2}{2p} \right)\right)\mathbb{E}[\|\nabla_y K_t\!-\!w_t\|^2]\\
        &-24r\rho \mathbb{E}[\|x^*(z_t)-x(y_+^t(z_t),z_t)\|^2].
\end{aligned}
\]
Before further bound it, we first denote $s_x^+=2L\left(1+ 24\omega^2r\rho \sigma_2^2 \eta_y^2 L+\eta_y^2L\omega^2(s_y-\frac{5(1-p)\gamma L^2}{2p})\right)$ and $s_y^+= s_y-\frac{5\left(1-p\right)\gamma L^2}{2p} $. Then, using $\|x_t-x_+(y_t,z_t)\|^2\leq\ 2\|x_{t+1}-x_t\|^2+2\eta_x^2\|\nabla_x K_t-v_t\|^2$, 
we have 
\[
\begin{aligned}
     &\mathbb{E} [\Phi_t-\Phi_{t+1}]\\
        \geq\ &\left(s_x\!-\!\frac{(1\!-p)\gamma \left(3\left(L\!+\!r\right)^2\!+\!2L^2\right)}{2p}\!-\!2s_x^+\right) \mathbb{E}[\|x_{t+1}\!-x_t\|^2]\\
        &+ \frac{1}{2}s_y^+\mathbb{E}[\|y_+^t(z_t)-y_t\|^2] \\
        &+ \left(s_z-\frac{3\left(1-p\right)\gamma r^2}{2p}\right)\mathbb{E}[\|z_t-z_{t+1}\|^2] \\ 
        &+\left( \frac{\gamma}{2}-s_v-2\eta_x^2 s_x^+\right)\mathbb{E}[\|\nabla_x K_t-v_t\|^2] \\
        &+\left( \frac{\gamma}{2}-s_w-2\eta^2_y s_y^+\right)\mathbb{E}[\|\nabla_y K_t-w_t\|^2]\\
        &-24r\rho \mathbb{E}[\|x^*(z_t)-x(y_+^t(z_t),z_t)\|^2].
\end{aligned}
\]
\par Now, we are ready to bound the coefficients. First, since $ 2L\leq r\leq 4L$, then we have $\sigma_1\leq 2$ and $ \sigma_2\leq 3$. Therefore, we get that $L_d=L\left(1+\sigma_2\right)\leq 4L$. 
\begin{itemize}
    
\item  For coefficient of $\|z_t-z_{t+1}\|^2$, since $\rho$ satisfies $\rho\leq \frac{4p}{1200p+9r \gamma}$, then due to $\sigma_1\leq 2$, we have
\[
\begin{aligned}
    &s_z-\frac{3\left(1-p\right)\gamma r^2}{2p} \\
    =\ & \frac{5r}{6\rho}-\frac{r}{2}-2r\sigma_1-48r\rho \sigma_1^2-\frac{3\left(1-p\right)\gamma r^2}{2p}\\
    \geq\ & \frac{5r}{6\rho}-\frac{393r}{2}-\frac{3\gamma r^2}{2p}\geq \frac{r}{6\rho}.
\end{aligned}
\]
\item Armed with this bound, we can provide the upper bound for $s_w$: 
\begin{align}\label{sw_val}
    s_w=\ & 1+48r\rho \eta_y^2\sigma_2^2\leq 1+\frac{4r\rho}{3 L^2}\leq 1+\frac{1}{25L},
\end{align}
where the first inequality is due to $\eta_y \leq \frac{p}{p(2+18L)}\leq \frac{1}{18L}$ and the second inequality is from $\rho\leq \frac{4p}{1200p+9r\gamma}\leq \frac{1}{300}$ and $r\leq 4L$.
\item Due to $\eta_y\leq \frac{1}{4L(1+\omega)^2}$, we have  $\frac{1}{\eta_y}-L (1+\omega)^2 \geq \frac{3}{4 \eta_y}$.
In addition, we set that $\eta_y \leq \frac{p}{p\left(2+18L\right)+10\gamma L^2}$, then the following holds:
    \begin{align*}
    s_y^+=\ & \frac{1}{\eta_y}-\frac{L+1+2L_d}{2}-L(1+\omega)^2-\frac{5\gamma \left(1-p\right)}{2p}L^2\\
    \geq\  & \frac{3}{4\eta_y}-\frac{1+9L}{2}-\frac{5\gamma }{2p}L^2 \geq \frac{1}{2\eta_y}.
\end{align*}

\item Now, we consider the coefficient of $\|\nabla_y K_t-w_t\|^2$. Recall that $\gamma = 4+\frac{2}{L}$, we have 
\[
\begin{aligned}
    &\frac{\gamma}{2}-s_w-2\eta^2_y \left(s_y-\frac{5\left(1-p\right)\gamma L^2}{2p} \right) \\
    \geq \ & \frac{\gamma}{2}-1-\frac{1}{25L} -2\eta_y\\
    \geq\ & \frac{\gamma}{2}-1-\frac{1}{10L}-\frac{2}{18L} \\
    \geq\ & \frac{\gamma}{2}-1-\frac{1}{2L}\geq \frac{\gamma}{4},
\end{aligned}
\]
where the second inequality is due to the fact that $\eta_y \leq \frac{p}{p(2+18L)}\leq \frac{1}{18L}$.

\item With above lower bounds, we are ready to bound $s_x^+$ as follows:
\begin{align}\label{sx_plus}
s_x^+=\ & 2L\left(1+24\omega^2r\rho \sigma_2^2 \eta_y^2 L+\eta_y^2L\omega^2s_y^+\right)\notag\\
\leq \ & 2L\left(1+24\left(1+\omega\right)^2r\rho\eta_y^2L\sigma_2^2+\left(1+\omega\right)^2\eta_y L\right)\notag\\
\leq\ &  2 L \left(1+\frac{1}{25}+\frac{1}{2}\right)\leq  \frac{9L}{2},
\end{align}
where the second inequality is due to $L(1+\omega)^2\leq \frac{1}{4\eta_y}$ and
\[
\begin{aligned}
    24\left(1+\omega\right)^2\eta_y^2Lr\rho\sigma_2^2\leq 6\eta_y r\rho \sigma_2^2\leq \frac{3r \rho}{L}\leq \frac{1}{25}.
\end{aligned}
\]
\item With the bounds for $s_x^+$ and $s_w$, and the step-size $\eta_x\leq \frac{p}{p(1+24L+2L^2)+80L^2\gamma}$, we find 
\[
\begin{aligned}
    &s_x-\frac{\left(1-p\right)\gamma \left(3\left(L+r\right)^2+2L^2\right)}{2p}-2s_x^+\\
    \geq\ & \frac{1}{\eta_x}-\frac{5L+1}{2}-L^2-\frac{80L^2\gamma}{2p}-9L\\
    \geq\ & \frac{1}{\eta_x}-\frac{1}{2}-12L-L^2-\frac{80L^2\gamma}{2p}\geq\frac{1}{2\eta_x}.
\end{aligned}
\]
\item Notice that $\eta_x\leq \frac{1}{24L}$, it allows us to bound $s_v$ as follows:
\begin{align}\label{sv_val}
    s_v=\ &\frac{1}{2}+\frac{2 \eta_x^2 L}{(1+\omega)^2}\leq \frac{1}{2}+2\eta_x^2L\leq \frac{1}{2}+\frac{2}{576L}\leq \frac{3}{5}.
    \end{align}
\item Then, the coefficient of $\|\nabla_x K_t-v_t\|^2$ can be bounded by 
\[
\begin{aligned}
     \frac{\gamma}{2}-s_v-2\eta_x^2 s_x^+\geq \ & \frac{\gamma}{2}-\frac{3}{5}-9\eta_x^2 L\\
     \geq \ &\frac{\gamma}{2}-\frac{3}{5}- \frac{9}{576L}\\
     \geq \ &\frac{\gamma}{2}-\frac{4}{5}\geq \frac{\gamma}{4}.
\end{aligned}
\]
Here, the first inequality is due to $s_x^+\leq \frac{9L}{2}$, the second inequality is from $\eta_x\leq \frac{1}{24L}$, and the last inequality is derived by $\gamma\geq 4$.
\end{itemize}
\par Combining the above pieces, we have
\begin{align*}
    &\mathbb{E}[\Phi_t-\Phi_{t+1}]\\
    \geq \ & \frac{1}{2\eta_x}\mathbb{E} [\|x_{t+1}-x_t\|^2]+\frac{1}{4\eta_y} \mathbb{E}[\|y_t-y_+^t(z_t)\|^2]\\
    &+\frac{r}{6\rho}\mathbb{E} [\|z_t-z_{t+1}\|^2]+\frac{\gamma}{4}\mathbb{E}[\|\nabla_x K_t-v_t\|^2] \notag\\
        &+\frac{\gamma}{4}\mathbb{E}[\|\nabla_y K_t-w_t\|^2]\\
        &-24r\rho  \mathbb{E}[\|x^*(z_t)-x(y_+^t(z_t),z_t)\|^2],
\end{align*}  
which completes the proof.
\end{proof}
\section{Descent of gradient estimators in ZeroSARAH-SGDA}\label{des_v_nab_proof}
\begin{proof}
Following the proof of \cite[Lemma 2]{Li2021ZeroSARAHEN}, we establish the descent property of gradient estimators. By the update of $v_t$ in \eqref{v_up_rul}, we have
    \begin{align*}
        &\mathbb{E}[\|v_t-\nabla_x K_t\|^2]\\
        =\ & \mathbb{E}\bigg[\bigg\|\frac{1}{b}\sum_{i\in B_t} \left(\nabla_x K_{i,t}-\nabla_x K_{i,t-1}\right)+(1-\lambda)v_{t-1}\!-\!\nabla_x K_t\\
        \ &+\!\lambda\left(\frac{1}{b}\!\sum_{i\in B_t}\!(\nabla_x K_{i,t-1}\!-\!d_{i,t-1})\!+\!\frac{1}{n}\sum_{i=1}^n\! d_{i,t-1}\right)\bigg\|^2\bigg]\\
        =\ &\mathbb{E}\bigg[\bigg\|\frac{1}{b}\sum_{i\in B_t}  \left(\nabla_x K_{i,t}-\nabla_x K_{i,t-1}\right)+\nabla_x K_{t-1}-\nabla_x K_t\\
        &+\!\lambda\Big(\!\frac{1}{b}\!\sum_{i\in B_t}\!\left(\nabla_x K_{i,t-1}\!-\!d_{i,t-1}\right)\!+\!\frac{1}{n}\!\sum_{i=1}^{n}d_{i,t-1}\!-\!\nabla_x K_{t-1}\!\Big)\!\bigg\|^2\bigg]\\
        &+(1-\lambda)^2\mathbb{E}[\|v_{t-1}-\nabla_x K_{t-1}\|^2]\\
        \leq \ &2 \mathbb{E}\bigg[\Big\|\frac{1}{b}\sum_{i\in B_t}  \left(\nabla_x K_{i,t}-\nabla_x K_{i,t-1}\right)+\nabla_x K_{t-1}-\nabla_x K_t\Big\|^2\bigg]\\
        &+2\mathbb{E}\bigg[\lambda^2\Big\|\frac{1}{b}\sum_{i\in B_t}\left(\nabla_x K_{i,t-1}-d_{i,t-1}\right)\\
        &\qquad +\frac{1}{n}\sum_{i=1}^{n}d_{i,t-1}-\nabla_x K_{t-1}\Big\|^2\bigg]\\
        &+(1-\lambda)^2\mathbb{E}[\|v_{t-1}-\nabla_x K_{t-1}\|^2]\\
        \leq \ & \frac{2\lambda^2}{b} \mathbb{E}\left[\frac{1}{n}\sum_{i=1}^n\|\nabla_x K_{i,t-1}-d_{i,t-1}\|\right]^2\\
        &+(1-\lambda)^2\mathbb{E}[\|v_{t-1}-\nabla_x K_{t-1}\|^2]\\
        &+\frac{6(L+r)^2}{b}\mathbb{E}[\|x_{t}-x_{t-1}\|^2]+\frac{6L^2}{b}\mathbb{E}[\|y_t-y_{t-1}\|^2]\\
        &+\frac{6r^2}{b}\mathbb{E}[\|z_t-z_{t-1}\|^2],
    \end{align*}
    where the last inequality holds due to the smoothness of function $K$ and the fact that $\mathbb{E}[\|x-\mathbb{E}x\|^2]\leq \mathbb{E}[\|x\|^2]$ for any random variable $x$. 
    Hence, by rearrangement, the inequality \eqref{nabx_v_diff} holds when $0<\lambda<1$. 
     The inequality \eqref{naby_w_diff} holds by a similar analysis.
\end{proof}
Then, for the term $\mathbb{E}\left[\frac{1}{n}\sum_{i=1}^n\|\nabla_x K_{i,t}-d_{i,t}\|^2\right]$ in the upper bound of \eqref{nabx_v_diff}, we provide the theoretical proof of its descent property shown in Lemma \ref{nab_d_lem}.
\begin{proof}
    According to the proposed algorithm, we have
    \begin{align*}
        &\mathbb{E}\left[\frac{1}{n}\sum_{i=1}^n \|\nabla_x K_{i,t}-d_{i,t}\|^2\right]\\
        =\ &(1-\frac{b}{n})\frac{1}{n}\sum_{i=1}^n\mathbb{E}[\|\nabla_x K_{i,t}-d_{i,t-1}\|^2]\\
        =\ &(1-\! \frac{b}{n})\frac{1}{n}\sum_{i=1}^n\!\mathbb{E}[\|\nabla_x K_{i,t}\!-\!\nabla_x K_{i,t-1}\!+\!\nabla_x K_{i,t-1}\!-\!d_{i,t-1}\|^2]\\
        \leq \ & (1-\frac{b}{n})(1+\beta)\frac{1}{n}\sum_{i=1}^n\mathbb{E}[\|\nabla_x K_{i,t-1}\!-\!d_{i,t-1}\|^2]\\
        &+(1-\frac{b}{n})(1+\frac{1}{\beta})\mathbb{E}\big[3(L+r)^2\|x_{t}-x_{t-1}\|^2\\
        &+3L^2\|y_{t}-y_{t-1}\|^2+3r^2\|z_{t}-z_{t-1}\|^2\big],
    \end{align*}
    where $\beta>0$, the last inequality holds due to the smoothness of function $K$ and Young's inequality. Hence, by rearrangement, the inequality \eqref{nabx_d_diff} holds. Following a similar analysis, the inequality \eqref{naby_h_diff} also holds.
\end{proof}
\section{Descent of $\Phi$ in ZeroSARAH-SGDA}\label{Phi_SARAH_proof}
\begin{proof}
    \par Recall the modified potential function $\Phi_t$ in \eqref{phi_func_def2}. Using the results in \eqref{nabx_v_diff}, \eqref{naby_w_diff}, \eqref{yt_rel}, and \eqref{Evsub}, $\Phi_t-\Phi_{t+1}$ satisfies
    \begin{align}
        &\mathbb{E} [\Phi_t-\Phi_{t+1}]\notag\\
        \geq \ & \left(s_x\!-\!\frac{\gamma\left(6(L+r)^2\!+\!4L^2\right)}{b}\!-\!\tau \xi \left(3(L+r)^2+2L^2\right)\!-\!2 s_x^+ \!\right)\notag\\
        &\qquad \mathbb{E}[\|x_{t+1}-x_{t}\|^2] \notag\\
        &+\!\frac{s_y^+}{2}\mathbb{E}[\|y_{t}\!-y_+^t(z_t)\|^2]\!+\!\left(\!s_z\!-\!\frac{6\gamma r^2}{b}\!-\!3\tau \xi r^2\!\right) \mathbb{E}[\|z_{t+1}\!-\!z_{t}\|^2]\notag\\
        &+(\gamma \lambda-s_v-2\eta_x^2s_x^+)  \mathbb{E}[\|v_t-\nabla_x K_t\|^2] \notag\\
        &+\left(\gamma \lambda-s_w-2\eta_y^2 s_y^+\right)\mathbb{E}[\|w_t-\nabla_y K_t\|^2] \notag \\
    &+\left(\tau\left(1-\zeta\right)-\frac{2\gamma \lambda^2}{b}\right)\mathbb{E}\left[\frac{1}{n}\sum_{i=1}^n \|\nabla_x K_{i,t}-d_{i,t}\|^2\right]\notag\\
    &+\left(\tau\left(1-\zeta\right)-\frac{2\gamma \lambda^2}{b}\right)\mathbb{E}\left[\frac{1}{n}\sum_{i=1}^n \|\nabla_y K_{i,t}-h_{i,t}\|^2\right] \notag \\
    &-24r\rho \kappa \mathbb{E}[\|y_t-y_+^t(z_t)\|],
    \end{align}
    where $s_y^+= s_y-\frac{10\gamma L^2}{b}-5\tau \xi L^2 $ and $s_x^+=2L\left(1+ 24\omega^2r\rho \sigma_2^2 \eta_y^2 L+\eta_y^2L\omega^2(s_y-\frac{10\gamma L^2}{b}-5\tau \xi L^2 )\right)$.
\par Since $ 2L\leq r\leq 4L$, then we have $\sigma_1\leq 2$ and $ \sigma_2\leq 3$. Hence, we get that $L_d=L\left(1+\sigma_2\right)\leq 4L$. 
\begin{itemize}
\item  For coefficient of $\|z_t-z_{t+1}\|^2$, since $\rho$ satisfies $\rho\leq \frac{4b}{1200b+36r \gamma+18\tau r b(1+b)}$ and $\beta=\frac{1}{b}$, then due to $\sigma_1\leq 2$, we have
\[
\begin{aligned}
    &s_z-\frac{6\gamma r^2}{b}-3\tau \xi r^2 \\
    = \ & \frac{5r}{6\rho}-\frac{r}{2}-2r\sigma_1-48r\rho \sigma_1^2-\frac{6\gamma r^2}{b}-3\tau \xi r^2 \\
    \geq \ & \frac{5r}{6\rho}-\frac{393 r}{2}-\frac{6\gamma r^2}{b}-3 \tau r^2 (1+b)\geq \frac{r}{6\rho}.
    \end{aligned} 
\]
\item Due to $\eta_y\leq \frac{1}{4L(1+\omega)^2}$, we have $\frac{1}{\eta_y}-L (1+\omega)^2 \geq \frac{3}{4 \eta_y}$.
In addition, we set that $\eta_y \leq \frac{b}{b\left(2+18L+20\tau L^2 (1+b) \right)+40\gamma L^2}$, then the following holds:
    \begin{align*}
    s_y^+=\ & \frac{1}{\eta_y}\!-\!\frac{L+1+2L_d}{2}\!-\!L(1+\omega)^2\!-\!\frac{10\gamma L^2}{b}\!-5\tau L^2 \xi \\
    \geq\  & \frac{3}{4\eta_y}-\frac{1+9L}{2}-\frac{10\gamma L^2}{b}-5\tau L^2 \left(1+b\right) \geq \frac{1}{2\eta_y}.
\end{align*}
\item Now, we consider the coefficient of $\|\nabla_y K_t-w_t\|^2$. Recall that $s_w\leq 1+\frac{1}{25L}$ in \eqref{sw_val} and $\gamma = \frac{2}{\lambda}+\frac{2}{5\lambda L}$, we have 
\[
\begin{aligned}
    \gamma \lambda-s_w-2\eta^2_y s_y^+ \geq \ & \gamma \lambda-1-\frac{1}{25L} -2\eta_y\\
    \geq\ & \gamma \lambda-1-\frac{1}{25L}-\frac{1}{9L} \\
    \geq\ & \gamma \lambda-1-\frac{1}{5L}\geq \frac{\gamma \lambda}{2},
\end{aligned}
\]
where the second inequality is from $\eta_y \leq \frac{1}{18L}$.

\item With the bounds for $s_x^+$ in \eqref{sx_plus} and $s_w$, if we set $\eta_x\leq \frac{b}{b(1+24L+2L^2)+310L^2\gamma+160(1+b)bL^2\tau}$, we find 
\[
\begin{aligned}
    & s_x-\frac{\gamma\left(6(L+r)^2+4L^2\right)}{b}-2 s_x^+\\
    &-\tau (1-\frac{b}{n})(1+\frac{1}{\beta}) \left(3(L+r)^2+2L^2\right)\\
    \geq\ & \frac{1}{\eta_x}-\frac{5L+1}{2}-L^2-\frac{\gamma\left(6(L+r)^2+4L^2\right)}{b}\\
    &-\tau (1+b) \left(3(L+r)^2+2L^2\right)-9L\\
    \geq\ & \frac{1}{\eta_x}-\frac{1}{2}-12L-L^2-\frac{155\gamma L^2}{b}-(1+b)80\tau L^2\\
    \geq\ &\frac{1}{2\eta_x},
\end{aligned}
\]
where we use $\beta=\frac{1}{b}$ and $r\leq 4L$ in the second inequality.
\item Using $s_v\leq \frac{3}{5}$ in \eqref{sv_val}, the coefficient of $\|\nabla_x K_t-v_t\|^2$ can be bounded by 
\[
\begin{aligned}
     \gamma \lambda-s_v-2\eta_x^2 s_x^+
     \geq \ & \gamma \lambda-\frac{3}{5}- \frac{9}{576L}\\
     \geq \ &\gamma \lambda-\frac{4}{5}\geq \frac{\gamma \lambda}{2}.
\end{aligned}
\]
Here, the first inequality is due to  $s_x^+\leq \frac{9L}{2}$ and $\eta_x\leq \frac{1}{24L}$, and the last inequality is derived by $\gamma> \frac{8}{5\lambda}$.
\item We consider the coefficient of $\frac{1}{n}\sum_{i=1}^n\|\nabla_x K_{i,t}-d_{i,t}\|^2$.
With $\beta=\frac{1}{b}$ and $b=a\sqrt{n}$, $a\geq 2$, the $\tau$ satisfies
\begin{align*}
\tau=\ 2\gamma \lambda^2
>\ & \frac{\gamma\lambda^2}{\sqrt{n}(\frac{1}{2\sqrt{n}}+\frac{1}{n})}\\
>\ &\frac{2\gamma \lambda^2}{b\left(1-\frac{1}{\sqrt{n}}-(1-\frac{b}{n})(1+\beta)\right)}.
\end{align*}
Then, we have
\[
\begin{aligned}
    \tau\left(1-\zeta\right)-\frac{2\gamma \lambda^2}{b}= \ & \tau\left(1-(1-\frac{b}{n})(1+\beta)\right)-\frac{2\gamma \lambda^2}{b} \\
    \geq \ & \frac{\tau}{\sqrt{n}}.
\end{aligned}
\]
\item Given that $\gamma=\mathcal{O}(\frac{1}{\lambda})$ and $\tau=2\gamma \lambda^2$, it follows that $\tau=\mathcal{O}(\lambda)$. Since the upper bounds of the step sizes $\eta_x$, $\eta_y$ and $\rho$ adhere to the form $\mathcal{O}(\frac{1}{\frac{\gamma}{b}+\tau(1+b)})$, we choose $\lambda=\frac{1}{b}$ to ensure an appropriate range of step sizes when $b$ is large.
\end{itemize}

\par For simplicity, we use some notations to denote the positive constant coefficients of each terms,
\begin{align*}
    &\mathbb{E}[\Phi_t-\Phi_{t+1}] \\
    \geq \ & c_x \mathbb{E} [\|x_{t+1}-x_t\|^2]+c_y \mathbb{E}[\|y_t-y_+^t(z_t)\|^2]\\
    &+c_z\mathbb{E} [\|z_t-z_{t+1}\|^2] \\
        &+c_v\mathbb{E}[\|\nabla_x K (x_t,z_t;y_t)-v_t\|^2]\\
        &+c_w\mathbb{E}[\|\nabla_y K(x_t,z_t;y_t)-w_t\|^2]\\
        &+c_\tau \mathbb{E}\bigg[\frac{1}{n}\sum_{i=1}^n\|\nabla_x K_i(x_t,z_t;y_t)-d_{i,t}\|^2\\
        &+\frac{1}{n}\sum_{i=1}^n \|\nabla_y K_i(x_t,z_t;y_t)-h_{i,t}\|^2\bigg]\\
        &-24r\rho \kappa \mathbb{E}[\|y_t-y_+^t(z_t)\|],
\end{align*}
where $c_x=\frac{1}{2\eta_x}$, $c_y=\frac{1}{4\eta_y}$, $c_z=\frac{r}{6\rho}$, $c_v=c_w=\frac{\gamma \lambda}{2}$, $c_{\tau}=\frac{\tau}{\sqrt{n}}$.
    
\end{proof}

\end{appendices}

\bibliographystyle{IEEEtran}
\bibliography{refer}

\begin{thebibliography}{10}
\providecommand{\url}[1]{#1}
\csname url@samestyle\endcsname
\providecommand{\newblock}{\relax}
\providecommand{\bibinfo}[2]{#2}
\providecommand{\BIBentrySTDinterwordspacing}{\spaceskip=0pt\relax}
\providecommand{\BIBentryALTinterwordstretchfactor}{4}
\providecommand{\BIBentryALTinterwordspacing}{\spaceskip=\fontdimen2\font plus
\BIBentryALTinterwordstretchfactor\fontdimen3\font minus
  \fontdimen4\font\relax}
\providecommand{\BIBforeignlanguage}[2]{{%
\expandafter\ifx\csname l@#1\endcsname\relax
\typeout{** WARNING: IEEEtran.bst: No hyphenation pattern has been}%
\typeout{** loaded for the language `#1'. Using the pattern for}%
\typeout{** the default language instead.}%
\else
\language=\csname l@#1\endcsname
\fi
#2}}
\providecommand{\BIBdecl}{\relax}
\BIBdecl

\bibitem{Distributed_Sadd}
Y.~Huang, Z.~Meng, J.~Sun, and W.~Ren, ``A unified distributed method for
  constrained networked optimization via saddle-point dynamics,'' \emph{IEEE
  Transactions on Automatic Control}, vol.~69, no.~3, pp. 1818--1825, 2024.

\bibitem{ben2009robust}
A.~Ben-Tal, L.~El~Ghaoui, and A.~Nemirovski, \emph{Robust Optimization}.\hskip
  1em plus 0.5em minus 0.4em\relax Princeton university press, 2009.

\bibitem{delage2010distributionally}
E.~Delage and Y.~Ye, ``Distributionally robust optimization under moment
  uncertainty with application to data-driven problems,'' \emph{Operations
  Research}, vol.~58, no.~3, pp. 595--612, 2010.

\bibitem{kuhn2019wasserstein}
D.~Kuhn, P.~M. Esfahani, V.~A. Nguyen, and S.~Shafieezadeh-Abadeh,
  ``Wasserstein distributionally robust optimization: Theory and applications
  in machine learning,'' in \emph{Operations Research \& Management Science in
  the Age of Analytics}.\hskip 1em plus 0.5em minus 0.4em\relax Informs, 2019,
  pp. 130--166.

\bibitem{littman1994markov}
M.~L. Littman, ``Markov games as a framework for multi-agent reinforcement
  learning,'' in \emph{Machine Learning Proceedings 1994}.\hskip 1em plus 0.5em
  minus 0.4em\relax Elsevier, 1994, pp. 157--163.

\bibitem{dai2018sbeed}
B.~Dai, A.~Shaw, L.~Li, L.~Xiao, N.~He, Z.~Liu, J.~Chen, and L.~Song,
  ``{SBEED}: Convergent reinforcement learning with nonlinear function
  approximation,'' in \emph{Proceedings of the 35th International Conference on
  Machine Learning}, vol.~80.\hskip 1em plus 0.5em minus 0.4em\relax PMLR,
  10--15 Jul 2018, pp. 1125--1134.

\bibitem{zhang2021multi}
K.~Zhang, Z.~Yang, and T.~Ba{\c{s}}ar, ``Multi-agent reinforcement learning: A
  selective overview of theories and algorithms,'' \emph{Handbook of
  Reinforcement Learning and Control}, pp. 321--384, 2021.

\bibitem{SGDA}
T.~Lin, C.~Jin, and M.~Jordan, ``On gradient descent ascent for
  nonconvex-concave minimax problems,'' in \emph{Proceedings of the 37th
  International Conference on Machine Learning}, vol. 119, 2020, pp.
  6083--6093.

\bibitem{two_GDA}
M.~Heusel, H.~Ramsauer, T.~Unterthiner, B.~Nessler, and S.~Hochreiter, ``{GAN}s
  trained by a two time-scale update rule converge to a local nash
  equilibrium,'' in \emph{Advances in Neural Information Processing Systems},
  vol.~30.\hskip 1em plus 0.5em minus 0.4em\relax Curran Associates, Inc.,
  2017.

\bibitem{yang2020catalyst}
J.~Yang, S.~Zhang, N.~Kiyavash, and N.~He, ``A catalyst framework for minimax
  optimization,'' in \emph{Advances in Neural Information Processing Systems},
  vol.~33.\hskip 1em plus 0.5em minus 0.4em\relax Curran Associates, Inc.,
  2020, pp. 5667--5678.

\bibitem{alter_prox_nc_c}
R.~I. Bo\c{t} and A.~B\"{o}hm, ``Alternating proximal-gradient steps for
  (stochastic) nonconvex-concave minimax problems,'' \emph{SIAM Journal on
  Optimization}, vol.~33, no.~3, pp. 1884--1913, 2023.

\bibitem{zhang2022sapd}
X.~Zhang, N.~S. Aybat, and M.~Gurbuzbalaban, ``{SAPD+}: An accelerated
  stochastic method for nonconvex-concave minimax problems,'' in \emph{Advances
  in Neural Information Processing Systems}, vol.~35.\hskip 1em plus 0.5em
  minus 0.4em\relax Curran Associates, Inc., 2022, pp. 21\,668--21\,681.

\bibitem{spider_paper}
C.~Fang, C.~J. Li, Z.~Lin, and T.~Zhang, ``{SPIDER}: Near-optimal non-convex
  optimization via stochastic path-integrated differential estimator,'' in
  \emph{Advances in Neural Information Processing Systems}, vol.~31.\hskip 1em
  plus 0.5em minus 0.4em\relax Curran Associates, Inc., 2018.

\bibitem{NEURIPS2019_b8002139}
A.~Cutkosky and F.~Orabona, ``Momentum-based variance reduction in non-convex
  {SGD},'' in \emph{Advances in Neural Information Processing Systems},
  vol.~32.\hskip 1em plus 0.5em minus 0.4em\relax Curran Associates, Inc.,
  2019.

\bibitem{SREDA_nips2020}
L.~Luo, H.~Ye, Z.~Huang, and T.~Zhang, ``Stochastic recursive gradient descent
  ascent for stochastic nonconvex-strongly-concave minimax problems,'' in
  \emph{Advances in Neural Information Processing Systems}, vol.~33.\hskip 1em
  plus 0.5em minus 0.4em\relax Curran Associates, Inc., 2020, pp.
  20\,566--20\,577.

\bibitem{pmlr-v139-song21d}
C.~Song, S.~J. Wright, and J.~Diakonikolas, ``Variance reduction via
  primal-dual accelerated dual averaging for nonsmooth convex finite-sums,'' in
  \emph{International Conference on Machine Learning}, vol. 139.\hskip 1em plus
  0.5em minus 0.4em\relax PMLR, 18--24 Jul 2021, pp. 9824--9834.

\bibitem{Li2021ZeroSARAHEN}
\BIBentryALTinterwordspacing
Z.~Li and P.~Richt{\'a}rik, ``Zerosarah: Efficient nonconvex finite-sum
  optimization with zero full gradient computation,'' \emph{ArXiv}, vol.
  abs/2103.01447, 2021. [Online]. Available:
  \url{https://api.semanticscholar.org/CorpusID:232092630}
\BIBentrySTDinterwordspacing

\bibitem{zhang2020single}
J.~Zhang, P.~Xiao, R.~Sun, and Z.~Luo, ``A single-loop smoothed gradient
  descent-ascent algorithm for nonconvex-concave min-max problems,'' in
  \emph{Advances in Neural Information Processing Systems}, vol.~33.\hskip 1em
  plus 0.5em minus 0.4em\relax Curran Associates, Inc., 2020, pp. 7377--7389.

\bibitem{Lu_2020}
\BIBentryALTinterwordspacing
S.~Lu, I.~Tsaknakis, M.~Hong, and Y.~Chen, ``Hybrid block successive
  approximation for one-sided non-convex min-max problems: Algorithms and
  applications,'' \emph{IEEE Transactions on Signal Processing}, vol.~68, p.
  3676–3691, 2020. [Online]. Available:
  \url{http://dx.doi.org/10.1109/TSP.2020.2986363}
\BIBentrySTDinterwordspacing

\bibitem{li2023nonsmooth}
J.~Li, L.~Zhu, and A.~M.-C. So, ``Nonsmooth nonconvex-nonconcave minimax
  optimization: Primal-dual balancing and iteration complexity analysis,''
  \emph{arXiv:2209.10825}, 2022.

\bibitem{NEURIPS2023_a961dea4}
T.~Zheng, L.~Zhu, A.~M.-C. So, J.~Blanchet, and J.~Li, ``Universal gradient
  descent ascent method for nonconvex-nonconcave minimax optimization,'' in
  \emph{Advances in Neural Information Processing Systems}, vol.~36.\hskip 1em
  plus 0.5em minus 0.4em\relax Curran Associates, Inc., 2023, pp.
  54\,075--54\,110.

\bibitem{den_nonmini_nips}
W.~Xian, F.~Huang, Y.~Zhang, and H.~Huang, ``A faster decentralized algorithm
  for nonconvex minimax problems,'' in \emph{Advances in Neural Information
  Processing Systems}, vol.~34.\hskip 1em plus 0.5em minus 0.4em\relax Curran
  Associates, Inc., 2021, pp. 25\,865--25\,877.

\bibitem{NEURIPS2020_0cc6928e}
J.~Yang, N.~Kiyavash, and N.~He, ``Global convergence and variance reduction
  for a class of nonconvex-nonconcave minimax problems,'' in \emph{Advances in
  Neural Information Processing Systems}, vol.~33.\hskip 1em plus 0.5em minus
  0.4em\relax Curran Associates, Inc., 2020, pp. 1153--1165.

\bibitem{xu2024derivativefree}
Z.~Xu, Z.~Wang, J.~Shen, and Y.~Dai, ``Derivative-free alternating projection
  algorithms for general nonconvex-concave minimax problems,''
  \emph{arXiv.2108.00473}, 2024.

\bibitem{pmlr-v119-liu20j}
S.~Liu, S.~Lu, X.~Chen, Y.~Feng, K.~Xu, A.~Al-Dujaili, M.~Hong, and U.-M.
  O'Reilly, ``Min-max optimization without gradients: Convergence and
  applications to black-box evasion and poisoning attacks,'' in
  \emph{Proceedings of the 37th International Conference on Machine Learning},
  vol. 119, 2020, pp. 6282--6293.

\bibitem{Pang_VIp}
J.-S. Pang, ``A posteriori error bounds for the linearly-constrained
  variational inequality problem,'' \emph{Mathematics of Operations Research},
  vol.~12, no.~3, pp. 474--484, 1987.

\bibitem{page_p}
Z.~Li, H.~Bao, X.~Zhang, and P.~Richtarik, ``Page: A simple and optimal
  probabilistic gradient estimator for nonconvex optimization,'' in
  \emph{Proceedings of the 38th International Conference on Machine Learning},
  ser. Proceedings of Machine Learning Research, M.~Meila and T.~Zhang, Eds.,
  vol. 139.\hskip 1em plus 0.5em minus 0.4em\relax PMLR, 18--24 Jul 2021, pp.
  6286--6295.

\end{thebibliography}

\end{document}